\documentclass[10pt]{amsart}
      \usepackage{amsmath,amsfonts}
\def\rg{\hbox to 30pt{\rightarrowfill}}
\def\lg{\hbox to 30pt{\leftarrowfill}}

      \parskip\smallskipamount
          \newtheorem{theorem}{Theorem}[section]

      \newtheorem{lemma}[theorem]{Lemma}

      \makeatletter
      \@addtoreset{equation}{section}
      \makeatother

      \newcommand{\CC}{{\mathbb C}}
      \newcommand{\NN}{{\mathbb N}}
      \newcommand{\QQ}{{\mathbb Q}}
      
      \newcommand{\DD}{{\mathbb D}}
      
      \newcommand{\FF}{{\mathbb F}}
      \newcommand{\TT}{{\mathbb T}}

      \newcommand{\cA}{{\mathcal A}}
      
      \newcommand{\cC}{{\mathcal C}}
      \newcommand{\cD}{{\mathcal D}}
      \newcommand{\cE}{{\mathcal E}}

      \newcommand{\cH}{{\mathcal H}}
      \newcommand{\cI}{{\mathcal I}}
      \newcommand{\cK}{{\mathcal K}}
      
      \newcommand{\cM}{{\mathcal M}}
      \newcommand{\cN}{{\mathcal N}}
      
      \newcommand{\cQ}{{\mathcal Q}}
      \newcommand{\cP}{{\mathcal P}}
      \newcommand{\cR}{{\mathcal R}}

      \newcommand{\cU}{{\mathcal U}}
      \newcommand{\cV}{{\mathcal V}}

      \newcommand{\rank}{\hbox{\rm{rank}}\,}

      \newdimen\expt
      \def\boxit#1{\setbox0\hbox{$\displaystyle{#1}$}
            \hbox{\lower.4\expt
       \hbox{\lower3\expt\hbox{\lower\dp0
            \hbox{\vbox{\hrule height.4\expt
       \hbox{\vrule width.4\expt\hskip3\expt
            \vbox{\vskip3\expt\box0\vskip2\expt}%
       \hskip3\expt\vrule width.4\expt}\hrule height.4\expt}}}}}}
      
\begin{document}
       \pagestyle{myheadings}
      \markboth{ Gelu Popescu}{  Pluriharmonic functions on noncommutative polyballs  }
      %\pagestyle{plain}
      %\begin{flushright}
       % \it Date of this draft: \today
      %\end{flushright}
      %\bigskip

      \title [   And\^ o   dilations and inequalities on  noncommutative  domains ]
      {      And\^ o   dilations and inequalities on  noncommutative  domains }
        \author{Gelu Popescu}
    % \date{\today}
\date{November 14, 2016}
      \thanks{Research supported in part by  NSF grant DMS 1500922}

      \address{Department of Mathematics, The University of Texas
      at San Antonio \\ San Antonio, TX 78249, USA}
      \email{\tt gelu.popescu@utsa.edu}

\begin{abstract}
We obtain  intertwining dilation theorems for noncommutative regular domains $\cD_f$  and noncommutative varieties $\cV_J$ in $B(\cH)^n$, which generalize Sarason \cite{S} and Sz.-Nagy--Foia\c s \cite{SzF}  commutant lifting theorem for commuting contractions. We present several applications including  a new proof for the commutant lifting theorem for pure elements in the domain $\cD_f$ (resp. variety $\cV_J$) as well as a Schur    type representation   for the  unit ball of the Hardy algebra associated with the variety $\cV_J$.

We provide And\^ o type dilations and inequalities for bi-domains $\cD_f\times_c \cD_f$ which consist of all pairs $({\bf X}, {\bf Y})$ of tuples  ${\bf X}:=(X_{1},\ldots, X_{n_1})\in \cD_f$  and
      ${\bf Y}:=(Y_{1},\ldots, Y_{n_2})\in \cD_g$ which commute, i.e. each entry of ${\bf X}$ commutes with each entry of ${\bf Y}$. The results are new even when $n_1=n_2=1$. In this particular case, we obtain
       extensions of And\^ o's results \cite{An} and Agler-McCarthy's inequality \cite{AM}  for commuting contractions to  larger classes of commuting operators.

       All the results are extended to bi-varieties $\cV_{J_1}\times_c \cV_{J_2}$, where  $\cV_{J_1}$ and $\cV_{J_2}$ are noncommutative varieties   generated by  WOT-closed  two-sided ideals   in  noncommutative   Hardy algebras. The commutative case as well as the matrix case when $n_1=n_2=1$  are   also discussed.
 \end{abstract}

      \maketitle

\section*{Introduction}

  Extending  von Neumann \cite{vN} inequality for one contraction and Sz.-Nagy proof using dilation theory \cite{SzFBK-book}, And\^ o \cite{An} proved a dilation result that implies his celebrated inequality which says that if $T_1$ and $T_2$ are commuting contractions on a Hilbert space, then for any polynomial $p$ in two variables,
 $$
\|p(T_1, T_2)\|\leq \|p\|_{\DD^2},
$$
where $\DD^2$ is the bidisk in $\CC^2$. For a nice survey and further generalizations of these inequalities  we refer to Pisier's book \cite{Pi-book} (see also \cite{vN}, \cite{SzFBK-book}, \cite{NV}, \cite{Varo}, \cite{Po-von}, \cite{Pa-book}, \cite{CW1},   \cite{CW2}, \cite{AM}, and \cite{DS}).
Inspired by the work of Agler-McCarthy \cite{AM} and Das-Sarkar \cite{DS} on distinguished varieties and Ando's inequality for two commuting contractions,
  we found, in a very recent  paper \cite{Po-Ando},  analogues of And\^ o's results for the elements of the bi-ball ${\bf P}_{n_1,n_2}$ which consists of all pairs $({\bf X}, {\bf Y})$ of row contractions ${\bf X}:=(X_{1},\ldots, X_{n_1})$  and
      ${\bf Y}:=(Y_{1},\ldots, Y_{n_2})$ which commute, i.e. each entry of ${\bf X}$ commutes with each entry of ${\bf Y}$.
      The results were obtained in a more general setting, namely, when ${\bf X}$ and ${\bf Y}$ belong to  noncommutative varieties $\cV_1$ and $\cV_2$ determined by row contractions subject to constraints such as
$$q(X_1,\ldots, X_{n_1})=0 \quad \text{and} \quad r(Y_1,\ldots, Y_{n_2})=0, \qquad  q\in \cP, r\in \cR,
$$
 respectively, where $\cP$ and $\cR$ are sets of noncommutative  polynomials.
  This led to one of the main results of the paper, an And\^ o type inequality on noncommutative varieties, which,  in the particular case when  $n_1=n_2=1$ and $T_1$ and $T_2$ are commuting contractive matrices  with spectrum in the open unit disk $\DD:=\{z\in \CC:\ |z|<1\}$, takes the form
$$
\|p(T_1, T_2)\|\leq \min\left\{ \|p(B_1\otimes I_{\CC^{d_1}},\varphi_1(B_1))\|, \|p(\varphi_2(B_2), B_2\otimes I_{\CC^{d_2}})\|\right\},
$$
where $(B_1\otimes I_{\CC^{d_1}},\varphi_1(B_1))$ and $(\varphi_2(B_2), B_2\otimes I_{\CC^{d_2}})$ are analytic dilations of $(T_1, T_2)$ while $B_1$ and $B_2$ are the universal models  associated with $T_1$ and $T_2$, respectively. In this  setting, the inequality  is sharper than And\^ o's inequality and Agler-McCarthy's inequality \cite{AM}. We obtained  more general inequalities  for arbitrary commuting contractive matrices  and improve And\^ o's inequality for   commuting contractions  when at least one of them is of class $\cC_0$. In this setting, it would be interesting to find  good analogues for {\it distinguished varieties} in the sense of  \cite{AM}.

 Let $\FF_n^+$ be the unital free semigroup on $n$ generators
$g_1,\ldots, g_n$ and the identity $g_0$.  The length of $\alpha\in
\FF_n^+$ is defined by $|\alpha|:=0$ if $\alpha=g_0$ and
$|\alpha|:=k$ if
 $\alpha=g_{i_1}\cdots g_{i_k}$, where $i_1,\ldots, i_k\in \{1,\ldots, n\}$.
 If ${\bf Z}:=\left< Z_1,\ldots, Z_n\right>$  is an $n$-tuple of  noncommutative indeterminates, we use the notation $Z_\alpha:= Z_{i_1}\cdots Z_{i_k}$  and $Z_{g_0}:=1$.  We denote by $\CC\left<{\bf Z}\right>$ the complex algebra of all polynomials in $Z_1,\ldots, Z_n$.
A polynomial  $f:=\sum_{\alpha\in\FF_n^+} a_\alpha Z_\alpha$ in $\CC\left<{\bf Z}\right>$ is called  {\it positive regular} if  the coefficients satisfy the conditions: $a_\alpha\geq 0$ for any $\alpha\in \FF_n^+$, \ $a_{g_0}=0$, and
  $a_{g_i}>0$ if  $i=1,\ldots, n$.  Define  the {\it noncommutative regular  domain}
$$
\cD_f(\cH):=\left\{ {\bf X}:=(X_1,\ldots, X_n)\in B(\cH)^n: \
 \sum_{|\alpha|\geq 1} a_\alpha X_\alpha
X_\alpha^* \leq I_\cH\right\}
$$
and the {\it noncommutative
  ellipsoid} $\cE_f(\cH)\supseteq \cD_f(\cH)$ by setting
$$
\cE_f(\cH):=\left\{ {\bf X}:=(X_1,\ldots, X_{n})\in B(\cH)^n: \  \sum_{|\beta|=1} a_\beta X_\beta X_\beta^*\leq I_\cH\right\},
$$
where $B(\cH)$ stands for the algebra of all bounded linear operators on a Hilbert space $\cH$.
Given $n_1, n_2 \in\NN:=\{1,2,\ldots\}$ and   $\Omega_j\subseteq B(\cH)^{n_j}$, $j=1,2$, we denote by $\Omega_1\times_c\Omega_2$
   the set of all pairs  $ ({\bf X},{ \bf Y})\in \Omega_1\times  \Omega_2$
     with the property that the entries of ${\bf X}:=(X_{1},\ldots, X_{n_1})$  are commuting with the entries of
      ${\bf Y}:=(Y_{1},\ldots, Y_{n_2})$.

 The main goal of the present paper is to extend the results from \cite{Po-Ando} for bi-balls and obtain And\^ o type dilations and inequalities for    bi-domains and noncommutative varieties:
 $$
 \cD_f(\cH)\times_c \cD_g(\cH)\quad \text{ and } \quad \cV_{J_1}(\cH)\times_c\cV_{J_2}(\cH),
 $$
 where $f\in \CC\left<{\bf Z}\right>$ and $g\in \CC\left<{\bf Z}'\right>$ are positive regular  noncommutative polynomials while $\cV_{J_1}$ and $\cV_{J_2}$ are varieties generated by  WOT-closed  two-sided ideals  in certain noncommutative   Hardy algebras.

In Section 2, we obtain an intertwining dilation theorem for bi-domains which    generalizes  Sarason  \cite{S} and  Sz.-Nagy--Foia\c s \cite{SzF} commutant lifting theorem  for commuting contractions  in  the framework of noncommutative regular  domains and Poisson kernels on weighted Fock spaces (see \cite{Po-domains}). As a consequence, we obtain a new proof for the commutant lifting theorem for pure elements in $\cD_f$.

These results are extended, in Section 3, to noncommutative varieties $\cV_J\subseteq \cD_f$ which are generated by WOT-closed  two-sided ideals $J$   in the   Hardy algebra $F_n^\infty(\cD_f)$, a noncommutative multivariable  version of the classical Hardy algebra $H^\infty (\DD)$.
More precisely,   the noncommutative variety $\cV_{J}(\cH)$ is defined as the set of all   {\it pure}  $n$-tuples
${\bf X}:=(X_1,\ldots, X_n)\in \cD_f(\cH)$    with the property that
  $$
\varphi(X_1,\ldots, X_n)=0\quad \text{for any } \ \varphi\in J,
$$
where $\varphi(X_1,\ldots, X_n)$ is defined   using  the $F_n^\infty(\cD_f)$-functional calculus for pure elements in $\cD_f(\cH)$ (see \cite{Po-domains}). Each variety $\cV_J$  is associated with certain  universal models ${\bf B}=(B_1,\ldots, B_n)$  and ${\bf C}=(C_1,\ldots, C_n)$ of constrained   creation operators  acting on a  subspace $\cN_J$ of the full Fock space with $n$ generators $F^2(H_n)$.
The noncommutative Hardy algebras $F_n^\infty(\cV_{J})$ and  $R_n^\infty(\cV_{J})$  are   the $WOT$-closed algebras generated by
$I, B_1,\ldots, B_n$ and  $I, C_1,\ldots, C_n$, respectively.
Using our intertwining dilation theorem on noncommutative varieties, we obtain  a  Schur \cite{Sc} type representation   for the  unit ball of $\cR_{n}^\infty(\cV_J)\bar \otimes B(\cH', \cH)$.

In Section 4, we obtain And\^ o type dilations and inequalities  for   noncommutative varieties
$$\cV_{J_1}(\cH)\times_c\cV_{J_2}(\cH),
$$ where
 $\cV_{J_1}(\cH)\subseteq \cD_f(\cH)$ and $\cV_{J_2}(\cH)\subseteq \cD_g(\cH)$.
 We   prove that any pair $({\bf T}_1, {\bf T}_2)$ in
 $\cV_{J_1}(\cH)\times_c\cV_{J_2}(\cH)$
has {\it analytic dilations}
 $$({\bf B}_1\otimes  I_{\ell^2},  \varphi_1({\bf  C}_1))\quad \text{and} \quad
 (\varphi_2({\bf C}_2), {\bf B}_2\otimes I_{\ell^2})
 $$
 where
$\varphi_1({\bf  C}_1)$ and $\varphi_2({\bf  C}_2)$ are some     multi-analytic operators with  respect to the universal models ${\bf B}_1$ and ${\bf B}_2$ of the varieties $\cV_{J_1}$ and $\cV_{J_2}$, respectively.
As a consequence, we show that the inequality
$$
\|[p_{rs}({\bf T}_1,{\bf T}_2)]_{k}\|\leq \min \left\{ \|[p_{rs}({\bf B}_1\otimes I_{\ell^2},  \varphi_1({\bf  C}_1))]_{k}\|,  \|[p_{rs}({\varphi_2({\bf  C}_2), \bf B}_2\otimes I_{\ell^2})]_{k}\|\right\}
$$
holds
 for any  $[p_{rs}]_k\in M_k(\CC\left<{\bf Z}, {\bf Z}'\right>)$ and   $k\in \NN$.
 Here, $\CC\left<{\bf Z}, {\bf Z}'\right>$   denotes the complex algebra of all  polynomials in  noncommutative indeterminates
    ${\bf Z}:=\left< Z_1,\ldots, Z_{n_1}\right>$ and
${\bf Z}':=\left< Z_1',\ldots, Z_{n_2}'\right>$, where we    assume that
 $Z_iZ_j'=Z_j'Z_i$ for any $i\in \{1,\ldots, n_1\}$ and $j\in \{1,\ldots, n_2\}$.

 On the other hand, we  prove that the abstract bi-domain
$$
\cD_f\times_c \cE_g:=\{\cD_f(\cH)\times_c \cE_g(\cH): \ \cH \text{ is a Hilbert space}\}
$$
has a universal {\it analytic model} \  $({\bf W}_1\otimes I_{\ell^2},  \psi({\bf  \Lambda}_1))$,
 where the tuples ${\bf W}_1=(W_{1,1},\ldots, W_{1,n_1})$ and ${\bf \Lambda}_1=(\Lambda_{1,1},\ldots, \Lambda_{1,n_1})$ are the weighted left and right creation operators on the full Fock space $F^2(H_{n_1})$, respectively, associated with the regular domain $\cD_f$.
More precisely, we show that the inequality
$$
\|[p_{rs}({\bf T}_1,{\bf T}_2)]_{k}\|\leq  \|[p_{rs}({\bf W}_1\otimes I_{\ell^2},  \psi({\bf  \Lambda}_1))]_{k}\|, \qquad  [p_{rs}]_k\in M_k(\CC\left<{\bf Z}, {\bf Z}'\right>),
$$
holds  for any  $({\bf T}_1, {\bf T}_2)\in \cD_f(\cH)\times_c \cE_g(\cH)$ and any $k\in \NN$. A similar result holds for the abstract variety
$\cV_J\times_c \cE_g$.

We will see, in Section 4,  that all the results of the present paper concerning And\^ o type dilations and inequalities can be written in the commutative multivariable  setting in terms of analytic multipliers of certain Hilbert spaces of holomorphic functions. These results are new even when $n_1=n_2=1$. In  this particular case, we obtain  extensions of And\^ o's results for commuting contractions \cite{An}, Agler-McCarthy's inequality \cite{AM},  and Das-Sarkar extension \cite{DS}, to  larger classes of commuting operators.

Finally, we would like to thank the referee for helpful comments and suggestions on the paper.

\smallskip

\section{Preliminaries on noncommutative regular domains and universal models}

In this section, we recall from \cite{Po-domains} basic facts concerning the noncommutative regular domains $\cD_f(\cH)\subset B(\cH)^n$ generated by positive regular formal power series, their universal models,  and the Hardy algebras they generate. We mention that commutative domains generated by positive regular polynomials were first introduced in \cite{Pot} and further elaborated in \cite{BS} and in a series of papers by the author (see \cite{Po-domains} and the references there in).

Let $H_n$ be an $n$-dimensional complex  Hilbert space with orthonormal
      basis
      $e_1$, $e_2$, $\dots,e_n$, where $n\in\{1,2,\dots\}$.        We consider
      the full Fock space  of $H_n$ defined by
      $$F^2(H_n):=\bigoplus_{k\geq 0} H_n^{\otimes k},$$
      where $H_n^{\otimes 0}:=\CC 1$ and $H_n^{\otimes k}$ is the (Hilbert)
      tensor product of $k$ copies of $H_n$.
      Define the left creation
      operators $S_i:F^2(H_n)\to F^2(H_n), \  i=1,\dots, n$,  by
      $$
       S_i\varphi:=e_i\otimes\varphi, \quad  \varphi\in F^2(H_n),
      $$
      and  the right creation operators
      $R_i:F^2(H_n)\to F^2(H_n)$  by
      $
       R_i\varphi:=\varphi\otimes e_i$, \ $ \varphi\in F^2(H_n)$.
The noncommutative analytic Toeplitz   algebra   $F_n^\infty$
  and  its norm closed version,
  the noncommutative disc
 algebra  $\cA_n$,  were introduced by the author  (see \cite{Po-von}, \cite{Po-funct}, \cite{Po-analytic}) in connection
   with a multivariable noncommutative von Neumann inequality.
$F_n^\infty$  is the algebra of left multipliers of $F^2(H_n)$  and
can be identified with
 the
  weakly closed  (or $w^*$-closed) algebra generated by the left creation operators
   $S_1,\dots, S_n$  acting on   $F^2(H_n)$,
    and the identity.
     The noncommutative disc algebra $\cA_n$ is
    the  norm closed algebra generated by
   $S_1,\dots, S_n$,
    and the identity.

 A formal power series $f:=\sum_{\alpha\in\FF_n^+} a_\alpha Z_\alpha$ in noncommutative  indeterminates $Z_1,\ldots, Z_n$ is called  {\it positive regular} if  the coefficients satisfy the conditions: $a_\alpha\geq 0$ for any $\alpha\in \FF_n^+$, \ $a_{g_0}=0$,
  $a_{g_i}>0$ if  $i=1,\ldots, n$, and
$
\limsup_{k\to\infty} \left( \sum_{|\alpha|=k} |a_\alpha|^2\right)^{1/2k}<\infty.
$
If ${\bf X}:=(X_1,\ldots, X_n)\in B(\cH)^n$,    we set
$X_\alpha:= X_{i_1}\cdots X_{i_k}$ if $\alpha=g_{i_1}\cdots g_{i_k}$, where $i_1,\ldots, i_k\in \{1,\ldots, n\}$, and $X_{g_0}:=I_\cH$.
  Define  the noncommutative regular  domain
$$
\cD_f(\cH):=\left\{ {\bf X}:=(X_1,\ldots, X_n)\in B(\cH)^n: \
 \sum_{|\alpha|\geq 1} a_\alpha X_\alpha
X_\alpha^* \leq I_\cH\right\},
$$
where the convergence of the series is in the weak operator
topology.
The power series  $1-f $ is invertible with
its inverse
$g= \sum_{\alpha\in \FF_n^+} b_\alpha
  X_\alpha $,  $b_\alpha\in \CC$, satisfies the relation
\begin{equation*}
\begin{split}
g &= 1+ f + f^2+\cdots
=1+\sum_{m=1}^\infty \sum_{|\alpha|=m}\left(\sum_{j=1}^{|\alpha|}
\sum_{{\gamma_1\cdots \gamma_j=\alpha }\atop {|\gamma_1|\geq
1,\ldots,
 |\gamma_j|\geq 1}} a_{\gamma_1}\cdots a_{\gamma_j} \right)
  X_\alpha.
\end{split}
\end{equation*}
 Consequently, we have
\begin{equation}\label{b_alpha}
b_{g_0}=1 \quad \text{ and }\quad b_\alpha= \sum_{j=1}^{|\alpha|}
\sum_{{\gamma_1\cdots \gamma_j=\alpha }\atop {|\gamma_1|\geq
1,\ldots, |\gamma_j|\geq 1}} a_{\gamma_1}\cdots a_{\gamma_j}   \quad
\text{ if } \ |\alpha|\geq 1.
\end{equation}

The {\it weighted left creation  operators} $W_i:F^2(H_n)\to
F^2(H_n)$, $i=1,\ldots, n$,  associated with the
 noncommutative domain $\cD_f$  are defined by setting $W_i:=S_iD_i$, where
 $S_1,\ldots, S_n$ are the left creation operators on the full
 Fock space $F^2(H_n)$ and each diagonal operator $D_i:F^2(H_n)\to F^2(H_n)$,
is given by
$$
D_ie_\alpha:=\sqrt{\frac{b_\alpha}{b_{g_i \alpha}}} e_\alpha,\qquad
 \alpha\in \FF_n^+.
$$
 Note that
\begin{equation*}
W_\beta e_\gamma= \frac {\sqrt{b_\gamma}}{\sqrt{b_{\beta \gamma}}}
e_{\beta \gamma} \quad \text{ and }\quad W_\beta^* e_\alpha
=\begin{cases} \frac {\sqrt{b_\gamma}}{\sqrt{b_{\alpha}}}e_\gamma&
\text{ if }
\alpha=\beta\gamma \\
0& \text{ otherwise }
\end{cases}
\end{equation*}
 for any $\alpha, \beta \in \FF_n^+$.
 We prove in \cite{Po-domains} that
  ${\bf W}:=(W_1,\ldots, W_n)$ is a pure $n$-tuple in  $ \cD_f(F^2(H_n))$,
  i.e. $\Phi_{f,{\bf W}}(I)\leq I$ and
   $\Phi_{f,{\bf W}}^k(I)\to 0$ in the strong operator topology, as
   $k\to\infty$,
   where $\Phi_{f,{\bf W}}(Y):=\sum\limits_{|\alpha|\geq 1} a_\alpha W_\alpha
YW_\alpha^*$ for  $Y\in B(F^2(H_n))$  and the convergence is in the
weak operator topology. The $n$-tuple ${\bf W}$ plays the role of universal model for the noncommutative domain $\cD_f$.

We  also define the {\it weighted right creation operators}
$\Lambda_i:F^2(H_n)\to F^2(H_n)$ by setting
$\Lambda_i:= R_i G_i$, $i=1,\ldots, n$,  where $R_1,\ldots, R_n$ are
 the right creation operators on the full Fock space $F^2(H_n)$ and
 each  diagonal operator $G_i$  is defined by
$$
G_ie_\alpha:=\sqrt{\frac{b_\alpha}{b_{ \alpha g_i}}} e_\alpha,\qquad
 \alpha\in \FF_n^+,
$$
where the coefficients $b_\alpha$, $\alpha\in \FF_n^+$, are given
by relation \eqref{b_alpha}.
In this case, we have
\begin{equation*}
\Lambda_\beta e_\gamma=
\frac {\sqrt{b_\gamma}}{\sqrt{b_{ \gamma \tilde\beta}}}
e_{ \gamma \tilde \beta} \quad
\text{ and }\quad
\Lambda_\beta^* e_\alpha =\begin{cases}
\frac {\sqrt{b_\gamma}}{\sqrt{b_{\alpha}}}e_\gamma& \text{ if }
\alpha=\gamma \tilde \beta \\
0& \text{ otherwise }
\end{cases}
\end{equation*}
 for any $\alpha, \beta \in \FF_n^+$, where $\tilde \beta$ denotes
 the reverse of $\beta=g_{i_1}\cdots g_{i_k}$, i.e.
 $\tilde \beta:=g_{i_k}\cdots g_{i_1}$.
 We remark that if  $f :=\sum_{|\alpha|\geq 1} a_\alpha
Z_\alpha$ is a positive regular  power series,   then so is $\tilde{f} :=\sum_{|\alpha|\geq 1} a_{\tilde \alpha} Z_\alpha$.  Moreover,  ${\bf \Lambda}:=(\Lambda_1,\ldots, \Lambda_n)\in \cD_{\tilde f}(F^2(H_n))$ and $W_i=U^*\Lambda_i U$, where $U\in B(F^2(H_n))$ is the unitary operator defined by $U e_\alpha:=e_{\tilde \alpha}$, $\alpha\in \FF_n^+$. Throughout this paper, we will refer to the $n$-tuples ${\bf W}:=(W_1,\ldots, W_n)$  and
${\bf \Lambda}:=(\Lambda_1,\ldots, \Lambda_n)$ as the weighted creation operators associated with the regular  domain $\cD_f$.

In \cite{Po-domains}, we introduced the  domain  algebra $\cA_n(\cD_f)$ associated with the
noncommutative domain $\cD_f$ to be the norm closure of all
polynomials in the weighted left creation operators  $W_1,\ldots,
W_n$ and the identity.
Using the weighted right creation operators
 associated with $\cD_f$, one can   define    the corresponding
  domain algebra ${\cR}_n(\cD_f)$.
  The Hardy algebra $F_n^\infty(\cD_f)$ (resp. $R_n^\infty(\cD_f)$) is the $w^*$- (or WOT-,
SOT-) closure of all polynomials   in $W_1,\ldots, W_n$ (resp. $\Lambda_1,\ldots, \Lambda_n$) and the
identity.
  We proved  that  $F_n^\infty(\cD_f)'= R_n^\infty(\cD_f)$ and $R_n^\infty(\cD_f)'=F_n^\infty(\cD_f)$, where $'$ stands for the commutant.

Now, we recall (\cite{Po-poisson}, \cite{Po-domains}) some basic facts concerning the noncommutative Poisson kernels associated with the regular domains.
Let ${\bf T}:=(T_1,\ldots, T_n)$ be an $n$-tuple of operators in the
noncommutative domain $ \cD_f(\cH)$, i.e.
$\sum\limits_{|\alpha|\geq 1} a_\alpha T_\alpha T_\alpha^*\leq
I_\cH$.
 Define the positive linear mapping
$$\Phi_{f,{\bf T}}:B(\cH)\to
B(\cH) \quad \text{  by  } \quad
\Phi_{f,{\bf T}}(X)=\sum\limits_{|\alpha|\geq 1} a_\alpha T_\alpha
XT_\alpha^*,
$$
where the convergence is in the weak operator topology. We use the notation $\Phi_{f,{\bf T}}^m$ for the composition $\Phi_{f,{\bf T}}\circ \cdots \circ \Phi_{f,{\bf T}}$  of $\Phi_{f,{\bf T}}$ by itself $m$ times.
Since $\Phi_{f,{\bf T}}(I)\leq I$ and $\Phi_{f,{\bf T}}(\cdot)$ is a positive linear map,
 it is easy to see that $\{\Phi_{f,{\bf T}}^m(I)\}_{m=1}^\infty$ is a decreasing
  sequence of positive operators and, consequently,
$Q_{f,{\bf T}}:=\text{\rm SOT-}\lim\limits_{m\to\infty} \Phi_{f,{\bf T}}^m(I)$ exists.
We say that ${\bf T}$  is a  {\it pure}  $n$-tuple in $\cD_f(\cH)$  if
~$\text{\rm SOT-}\lim \limits_{m\to\infty} \Phi_{f,{\bf T}}^m(I)=0$.
Note that, for any ${\bf T}:=(T_1,\ldots, T_n)\in \cD_f(\cH)$ and $0\leq
r<1$, the $n$-tuple $r{\bf T}:=(rT_1,\ldots, rT_n)\in \cD_f(\cH)$ is  pure. Indeed, it is enough to see  that
$\Phi_{f,r{\bf T}}^m(I)\leq r^m \Phi_{f,{\bf T}}^m(I)\leq r^m I$ for any  $m\in
\NN$. Note also that if $\|\Phi_{f,{\bf T}}(I)\|<1$, then $T$ is pure. This is due to the fact that
$\|\Phi_{f,{\bf T}}^m(I)\|\leq \|\Phi_{f,{\bf T}}(I)\|^m$.
We define the noncommutative
Poisson kernel associated with the $n$-tuple ${\bf T} \in \cD_f(\cH)$ to be  the operator $K_{f,{\bf T}}:\cH\to
F^2(H_n)\otimes \cD_{\bf T}$ defined by
\begin{equation*}
 K_{f,{\bf T}}h=\sum_{\alpha\in \FF_n^+} \sqrt{b_\alpha}
e_\alpha\otimes \Delta_{f,{\bf T}} T_\alpha^* h,\qquad h\in \cH,
\end{equation*}
where
$\Delta_{f,{\bf T}}:=\left(I- \Phi_{f,{\bf T}}(I) \right)^{1/2}$ is the defect operator associated with ${\bf T}$ and  $\cD_{\bf T}:=\overline{\Delta_{f,{\bf T}}(\cH)}$ is the corresponding  defect space.
The operator
  $K_{f,{\bf T}}$ is a contraction  satisfying relation
$
  K_{f,{\bf T}}^* K_{f,{\bf T}}=I_\cH-Q_{f,{\bf T}}
$
and
\begin{equation}
\label{ker-inter}
K_{f,{\bf T}} T_i^*=(W_i^*\otimes I_{\cD_{\bf T}})K_{f,{\bf T}},\quad i=1,\ldots, n,
\end{equation}
where ${\bf W}:=(W_1,\ldots, W_n)$ is the universal  model
  associated with the noncommutative  regular domain $\cD_f$.
Moreover, $K_{f,{\bf T}}$ is an isometry if and only if ${\bf T} $  is  pure element of $ \cD_f(\cH)$.

\smallskip

\section{Intertwining dilation theorem  on  noncommutative bi-domains}

 In this section, we obtain an intertwining dilation theorem which    generalizes  Sarason and  Sz.-Nagy--Foia\c s commutant lifting theorem  for commuting contractions  in  the framework of noncommutative regular  domains and Poisson kernels on weighted Fock spaces. As a consequence, we obtain a new proof for the commutant lifting theorem for pure elements in $\cD_f$.
 More   applications of this result  will be considered in the next sections.

Unless otherwise specified, we assume, throughout this paper, that $f$ and $g$ are two  positive regular polynomials in  noncommutative indeterminates ${\bf Z}:=\left< Z_1,\ldots, Z_{n_1}\right>$ and  ${\bf Z}':=\left< Z_1',\ldots, Z_{n_2}'\right>$, respectively, of the form
 $$f:=\sum_{{\alpha\in \FF_{n_1}^+}, {  1\leq |\alpha|\leq k_1}}a_\alpha Z_\alpha\quad \text{ and } \quad
g:=\sum_{{\beta\in \FF_{n_2}^+},{ 1\leq |\beta|\leq k_2}}c_\beta Z_\beta'.
$$
Fix  two tuples of operators
  ${\bf T}_1=(T_{1,1},\ldots, T_{1,n_1})\in \cD_{f}(\cH)$ and   ${\bf T}'_1=(T_{1,1}',\ldots, T_{1,n_1}')\in\cD_{f}(\cH')$  and
  let ${\bf T}_2:=(T_{2,1},\ldots, T_{2,n_2})$, with $T_{2,j}:\cH'\to \cH$,  be such that ${\bf T}_2\in \cD_{g}(\cH', \cH)$ and   intertwines ${\bf T}_1$ with ${\bf T}_1'$, i.e.
   $$
   T_{2,j}T_{1,i}'=T_{1,i}T_{2,j}
   $$
   for any $i\in \{1,\ldots, n_1\}$ and $j\in \{1,\ldots, n_2\}$. We denote by $\cI({\bf T}_1,{\bf T}_1')$ the set of all   intertwining tuples  ${\bf T}_2$  of ${\bf T}_1$ and ${\bf T}_1'$.
   A straightforward calculation reveals that
$$
 \Delta_{f,{\bf T}_1}^2+ \Phi_{f, {\bf T}_1}(\Delta_{g,{\bf T}_2}^2)=\Phi_{g,{\bf T}_2}(\Delta_{f,{\bf T}_1'}^2)+\Delta_{g,{\bf T}_2}^2.
$$
 If
 the defect spaces $\cD_{{\bf T}_1}$,  $\cD_{{\bf T}_1'}$,  and $\cD_{{\bf T}_2}$ are finite dimensional with  dimensions $d_1:=\dim \cD_{{\bf T}_1}$, $d_1':=\dim \cD_{{\bf T}_1'}$, and  $d_2:=\dim \cD_{{\bf T}_2}$, and such that
 $
  d_1+m_1d_2=m_2 d_1'+d_2,
  $
  where
  $$
  m_i:=\text{\rm card} \{\alpha\in \FF_{n_j}^+: \ 1\leq|\alpha|\leq k_j\},\qquad j=1,2,
  $$
  then    there are unitary extensions $U:\cD_{{\bf T}_1}\oplus \bigoplus_{\alpha\in \FF_{n_1}^+, 1\leq |\alpha|\leq k_1}\cD_{{\bf T}_2}\to  \bigoplus_{\beta\in \FF_{n_2}^+, 1\leq |\beta|\leq k_2}\cD_{{\bf T}_1'}\oplus \cD_{{\bf T}_2}
$ of the isometry
\begin{equation}
\label{iso1}
U\left(\Delta_{{\bf T}_1}h\oplus \bigoplus_{{\alpha\in \FF_{n_1}^+, 1\leq |\alpha|\leq k_1}} \sqrt{a_\alpha}\Delta_{{\bf T}_2}T_{1,\alpha}^*h \right):=
\bigoplus_{\beta\in \FF_{n_2}^+, 1\leq |\beta|\leq k_2}\sqrt{c_\beta}\Delta_{{\bf T}_1'}T_{2,\beta}^*h ,\qquad h\in \cH.
\end{equation}
We denote by $\cU_{\bf T}$ the set of all  unitary extensions of the isometry given by relation \eqref{iso1}.
In case the above-mentioned dimensional conditions are not satisfied, then   let $\cK$ be  an infinite dimensional Hilbert space and note that
the operator defined by
\begin{equation}
\label{iso2}
U\left(\Delta_{{\bf T}_1}h\oplus \bigoplus_{{\alpha\in \FF_{n_1}^+, 1\leq |\alpha|\leq k_1}} [\sqrt{a_\alpha}\Delta_{{\bf T}_2}T_{1,\alpha}^*h\oplus 0] \right):=
\left(\bigoplus_{\beta\in \FF_{n_2}^+, 1\leq |\beta|\leq k_2}\sqrt{c_\beta}\Delta_{{\bf T}_1'}T_{2,\beta}^*h\right) \oplus 0
\end{equation}
is an isometry which can be extended to a unitary operator
$$U:\cD_{{\bf T}_1}\oplus \bigoplus_{\alpha\in \FF_{n_1}^+, 1\leq |\alpha|\leq k_1}(\cD_{{\bf T}_2}\oplus \cK)\to  \bigoplus_{\beta\in \FF_{n_2}^+, 1\leq |\beta|\leq k_2}\cD_{{\bf T}_1'}\oplus (\cD_{{\bf T}_2}\oplus \cK).
$$
 In this setting, we denote by $\cU_{\bf T}^\cK$ the set of all unitary extensions of the isometry defined by \eqref{iso2}.
Let
$ U=\left[ \begin{matrix} A&B\\C&D \end{matrix}\right]
$
be the operator matrix representation of $U\in \cU_{\bf T}^\cK$, where
\begin{equation}\label{ABCD}\begin{split}
 A&:\cD_{{\bf T}_1}\to \bigoplus_{{\beta\in \FF_{n_2}^+}, {1\leq |\beta|\leq k_2}}\cD_{{\bf T}_1'},\\
 B&:\bigoplus_{\alpha\in \FF_{n_1}^+, 1\leq |\alpha|\leq k_1}(\cD_{{\bf T}_2}\oplus \cK)\to \bigoplus_{\beta\in \FF_{n_2}^+, 1\leq |\beta|\leq k_2}\cD_{{\bf T}_1'},\\
 C&: \cD_{{\bf T}_1}\to \cD_{{\bf T}_2}\oplus \cK, \text{\rm and }\\
 D&: \bigoplus_{\alpha\in \FF_{n_1}^+, 1\leq |\alpha|\leq k_1}(\cD_{{\bf T}_2}\oplus \cK)\to \cD_{{\bf T}_2}\oplus \cK.
 \end{split}
 \end{equation}
 Given an operator $Z:\cN\to \cM$ and $n\in \NN$, we introduce the ampliation
$$
\text{\rm diag}_{n}(Z):=\left(\begin{matrix}
Z&\cdots &0\\
\vdots&\ddots& \vdots\\
0&\cdots&Z
\end{matrix} \right):\bigoplus_{s=1}^n \cN\to \bigoplus_{s=1}^n \cM.
$$
In what follows, we consider the lexicographic order  for  the free semigroup $\FF_{n_1}^+$, that is
$$
g_0<g_1<\cdots <g_{n_1}<
g_1g_1 <\cdots < g_1g_{n_1}< \cdots< g_{n_1}g_1<\cdots < g_{n_1}g_{n_1}< \cdots
$$
and so on. For the direct product $\FF_{n_1}^+\times \cdots \times \FF_{n_1}^+$ of $p$ copies of $\FF_{n_1}^+$, we say that $(\alpha_p,\ldots, \alpha_1)<(\beta_p,\ldots, \beta_1)$ if $\alpha_p<\beta_p$ or there is $i\in \{2,\ldots, p\}$ such that
$\alpha_p=\beta_p,\ldots, \alpha_i=\beta_i$, and $\alpha_{i-1}<\beta_{i-1}$.
We also use  the operator column notation
$\left[\begin{matrix}Y_{(\alpha_p,\ldots, \alpha_1)}\\
 \vdots\\
   \alpha_i\in \FF_{n_1}^+, 1\leq |\alpha_i|\leq k\end{matrix}\right]$, where
  the entries $Y_{(\alpha_p,\ldots, \alpha_1)}$ are  arranged in  the order mentioned above.
For simplicity, $[X_1,\ldots, X_n]$ denotes either the $n$-tuple $(X_1,\ldots, X_n)\in B(\cH)^n$ or the operator row matrix
$[X_1\cdots X_n]$ acting from $\cH^{(n)}$, the direct sum of $n$ copies of the Hilbert space $\cH$, to $\cH$.
\begin{lemma} \label{le1} If ${\bf X}_1:=[\sqrt{a_\alpha} T_{1,\alpha}:\ \alpha\in \FF_{n_1}^+, 1\leq |\alpha|\leq k_1 ]$  and $m_1:=\text{ \rm card} \{\alpha\in \FF_{n_1}^+: \ 1\leq |\alpha|\leq k_1\}$, then
\begin{equation*}
\begin{split}
&\text{\rm diag}_{m_1}\left(D\text{\rm diag}_{m_1}\left(\cdots
D\text{\rm diag}_{m_1}\left(\widehat\Delta_{{\bf T}_2}\right)
 {\bf X}_1^*\cdots\right){\bf X}_1^*\right){\bf X}_1^*\\
 &\qquad\qquad\qquad =
 \text{\rm diag}_{m_1}\left(D\text{\rm diag}_{m_1}\left(\cdots
D\text{\rm diag}_{m_1}\left(\widehat\Delta_{{\bf T}_2}\right)
 \cdots\right)\right)\left[\begin{matrix}\sqrt{a_{\alpha_1}\cdots a_{\alpha_p}}(T_{1,\alpha_p}\cdots T_{1,\alpha_1})^*\\
 \vdots\\
 \alpha_i\in \FF_{n_1}^+, 1\leq |\alpha_i|\leq k_1 \end{matrix}\right],
\end{split}
\end{equation*}
where $\text{\rm diag}_{m_1}$ appears  $p$ times on each side of the equality, and $\widehat\Delta_{{\bf T}_2}:=\left[\begin{matrix} \Delta_{{\bf T}_2}\\0\end{matrix}\right]:\cH\to \cD_{{\bf T}_2}\oplus \cK$.
\end{lemma}
\begin{proof}
Let $D=[D_{(\alpha)} : \ \alpha\in \FF_{n_1}^+, 1\leq |\alpha|\leq k_1]$ with $D_{(\alpha)}\in B(\cD_{{\bf T}_2}\oplus \cK)$ and note that
$$
D\text{\rm diag}_{m_1}\left(\widehat\Delta_{{\bf T}_2}\right)
 {\bf X}_1^*=\sum_{\alpha_1\in \FF_{n_1}^+, 1\leq |\alpha_1|\leq k_1} D_{(\alpha_1)}\widehat\Delta_{{\bf T}_2} \sqrt{a_{\alpha_1}}T_{1,\alpha_1}^*
 $$
and
$$
\text{\rm diag}_{m_1}\left(D\text{\rm diag}_{m_1}\left(\widehat\Delta_{{\bf T}_2}\right){\bf X}_1^*\right) {\bf X}_1^*
=
\left[\begin{matrix}\sum_{{\alpha_1\in \FF_{n_1}^+}\atop { 1\leq |\alpha_1|\leq k_1}} D_{(\alpha_1)}\widehat\Delta_{{\bf T}_2} \sqrt{a_{\alpha_1}}\sqrt{a_{\alpha_2}}(T_{1,\alpha_2}T_{1,\alpha_1})^*\\
\vdots\\
 \alpha_2\in \FF_{n_1}^+, 1\leq|\alpha_2|\leq k_1
\end{matrix}\right].
$$
An inductive argument shows that
\begin{equation*}
\begin{split}
&\text{\rm diag}_{m_1}\left(D\text{\rm diag}_{m_1}\left(\cdots
D\text{\rm diag}_{m_1}\left(\widehat\Delta_{{\bf T}_2}\right)
 {\bf X}_1^*\cdots\right){\bf X}_1^*\right){\bf X}_1^*\\
 &\qquad
 =
 \left[\begin{matrix}\sum_{{\alpha_1,\ldots, \alpha_{p-1}\in \FF_{n_1}^+}\atop { 1\leq |\alpha_i|\leq k_1}} D_{(\alpha_{p-1})}\cdots D_{(\alpha_1)}\widehat\Delta_{{\bf T}_2} \sqrt{a_{\alpha_1} \cdots a_{\alpha_{p-1}}a_{\alpha_{p}}}(T_{1,\alpha_p \alpha_{p-1}\cdots \alpha_1})^*\\
\vdots\\
 \alpha_p\in \FF_{n_1}^+, 1\leq|\alpha_p|\leq k_1
\end{matrix}\right],
\end{split}
\end{equation*}
where $\text{\rm diag}_{m_1}$ appears  $p$ times.
On the other hand, one can easily prove by induction that
\begin{equation*}
\begin{split}
& \text{\rm diag}_{m_1}\left(D\text{\rm diag}_{m_1}\left(\cdots
D\text{\rm diag}_{m_1}\left(\widehat\Delta_{{\bf T}_2}\right)
 \cdots\right)\right)\left[\begin{matrix}\sqrt{a_{\alpha_1}\cdots a_{\alpha_p}}(T_{1,\alpha_p}\cdots T_{1,\alpha_1})^*\\
 \vdots\\
 \alpha_i\in \FF_{n_1}^+, 1\leq |\alpha_i|\leq p \end{matrix}\right]\\
  &
 =\text{\rm diag}_{m_1}\left([D_{(\alpha_{p-1})}\cdots D_{(\alpha_{1})} \widehat\Delta_{{\bf T}_2}: \ \alpha_1,\ldots, \alpha_{p-1}\in \FF_{n_1}^+, 1\leq |\alpha_i|\leq k_1]\right) \\
 &\qquad \qquad\times
 \left[\begin{matrix}\left[\begin{matrix}  \sqrt{a_{\alpha_1} \cdots a_{\alpha_{p-1}}a_{\alpha_{p}}}(T_{1,\alpha_p \alpha_{p-1}\cdots \alpha_1})^*\\
\vdots\\
 \alpha_1,\ldots, \alpha_{p-1}\in \FF_{n_1}^+, 1\leq|\alpha_i|\leq k_1
 \end{matrix}\right]\\
 \vdots\\
 \alpha_p\in \FF_{n_1}^+, 1\leq|\alpha_p|\leq k_1
\end{matrix}\right].
 \end{split}
 \end{equation*}
 The proof is complete.
  \end{proof}

 \begin{lemma} \label{series} Let ${\bf T}_2\in \cI({\bf T}_1,{\bf T}_1')$
     and let
$ U=\left[ \begin{matrix} A&B\\C&D \end{matrix}\right]
$
be the matrix representation of a unitary extension  $U\in \cU_{\bf T}^\cK$ (see relation \eqref{ABCD}).
If
\begin{equation*}
\begin{split}
{\bf X}_1&:=[\sqrt{a_\alpha} T_{1,\alpha}:\ \alpha\in \FF_{n_1}^+, 1\leq |\alpha|\leq k_1 ],\\
 m_j&:=\text{ \rm card} \{\alpha\in \FF_{n_j}^+: \ 1\leq |\alpha|\leq k_j\}, \quad  j=1,2,
\end{split}
\end{equation*}
and ${\bf T}_1$ is a pure element in $\cD_f(\cH)$,  then
 \begin{equation*}
  \begin{split}
  \text{\rm diag}_{m_2}(\Delta_{{\bf T}_1'})&\left[\begin{matrix}
  \sqrt{c_\beta}T_{2,\beta}^*\\
  \vdots\\
  \beta\in \FF_{n_2}^+, 1\leq |\beta|\leq k_2\end{matrix}
  \right]\\
  &
 =A\Delta_{{\bf T}_1}h
+B\sum_{p=0}^\infty \text{\rm diag}_{m_1}\left(D\text{\rm diag}_{m_1}\left(\cdots
D\text{\rm diag}_{m_1}\left(C\Delta_{{\bf T}_1}\right) {\bf X}_1^*\cdots\right){\bf X}_1^*\right){\bf X}_1^*h,
 \end{split}
  \end{equation*}
for any $h\in \cH$, where $\text{\rm diag}_{m_1}$ appears $p+1$ times in the general term of  the series.
 \end{lemma}
 \begin{proof}
 Due to relation \eqref{iso2}, we have
\begin{equation}\label{A}
A\Delta_{{\bf T}_1}h +B\left[\begin{matrix}\left[\begin{matrix}
\Delta_{{\bf T}_2}\sqrt{a_\alpha} T_{1,\alpha}^*h\\0
\end{matrix}\right]\\
\vdots  \\
\alpha\in \FF_{n_1}^+, 1\leq |\alpha|\leq k_1 \end{matrix}\right]
=
\left[\begin{matrix}
  \Delta_{{\bf T}_1'}\sqrt{c_\beta}T_{2,\beta}^*\\
  \vdots\\
  \beta\in \FF_{n_2}^+, 1\leq |\beta|\leq k_2\end{matrix}
  \right]
\end{equation}
and
\begin{equation}\label{C}
C\Delta_{{\bf T}_1}h +D\left[\begin{matrix}\left[\begin{matrix}
\Delta_{{\bf T}_2}\sqrt{a_\alpha} T_{1,\alpha}^*h\\0
\end{matrix}\right]\\
\vdots  \\
\alpha\in \FF_{n_1}^+, 1\leq |\alpha|\leq k_1 \end{matrix}\right]
=
\left[\begin{matrix}
\Delta_{{\bf T}_2} h\\0
\end{matrix}\right]
\end{equation}
for any $h\in \cH$.
Since $\widehat\Delta_{{\bf T}_2}:=\left[\begin{matrix} \Delta_{{\bf T}_2}\\0\end{matrix}\right]:\cH\to \cD_{{\bf T}_2}\oplus \cK$, we can rewrite relations \eqref{A}  and \eqref{C} as
 \begin{equation}
 \label{AA}
 A\Delta_{{\bf T}_1}+B\text{\rm diag}_{m_1}(\widehat\Delta_{{\bf T}_2}) {\bf X}_1^*
 =
 \text{\rm diag}_{m_2}(\Delta_{{\bf T}_1'}) {\bf X}_2^*,
 \end{equation}
 where ${\bf X}_2:=[\sqrt{c_\beta} T_{2,\beta}:\ \beta\in \FF_{n_2}^+, 1\leq |\beta|\leq k_2 ]$,
and
\begin{equation}
 \label{CC}
 C\Delta_{{\bf T}_1}+D\text{\rm diag}_{m_1}(\widehat\Delta_{{\bf T}_2}) {\bf X}_1^*
 =
 \widehat\Delta_{{\bf T}_2},
 \end{equation}
respectively. Note that  using relation \eqref{CC}  we deduce that
\begin{equation}
\label{Di}
\text{\rm diag}_{m_1}\left(\widehat\Delta_{{\bf T}_2}\right){\bf X}_1^*=\text{\rm diag}_{m_1}\left(C\Delta_{{\bf T}_1}\right){\bf X}_1^*
+\text{\rm diag}_{m_1}\left(D\text{\rm diag}_{m_1}(\widehat\Delta_{{\bf T}_2}) {\bf X}_1^*\right){\bf X}_1^*,
\end{equation}
which combined with relation \eqref{AA} yields

\begin{equation*}
\begin{split}
\text{\rm diag}_{m_2}(\Delta_{{\bf T}_1'}) {\bf X}_2^*
=A\Delta_{{\bf T}_1}+B \text{\rm diag}_{m_1}\left(C\Delta_{{\bf T}_1}\right){\bf X}_1^*
+B\text{\rm diag}_{m_1}\left(D\text{\rm diag}_{m_1}(\widehat\Delta_{{\bf T}_2}) {\bf X}_1^*\right){\bf X}_1^*.
\end{split}
\end{equation*}
  Continuing to use relation \eqref{Di} in the latter relation  and the resulting ones, an induction argument leads to  the identity
  \begin{equation}\label{rel}
  \begin{split}
  \text{\rm diag}_{m_2}(\Delta_{{\bf T}_1'}) {\bf X}_2^*
=A\Delta_{{\bf T}_1}&+B \text{\rm diag}_{m_1}\left(C\Delta_{{\bf T}_1}\right){\bf X}_1^*\\
&+B\sum_{p=1}^m \text{\rm diag}_{m_1}\left(D\text{\rm diag}_{m_1}\left(\cdots
D\text{\rm diag}_{m_1}\left(C\Delta_{{\bf T}_1}\right) {\bf X}_1^*\cdots\right){\bf X}_1^*\right){\bf X}_1^* \\
&+B\text{\rm diag}_{m_1}\left(D\text{\rm diag}_{m_1}\left(\cdots
D\text{\rm diag}_{m_1}\left(\widehat\Delta_{{\bf T}_2}\right)
 {\bf X}_1^*\cdots\right){\bf X}_1^*\right){\bf X}_1^*,
  \end{split}
  \end{equation}
where $\text{\rm diag}_{m_1}$ appears $p+1$ times in the general term of  the sum above and $m+2$ times in the last term.
Since $\Delta_{{\bf T}_2}$ and $D$ are contractions and due to  Lemma \ref{le1}, one can easily see that
\begin{equation*}
\begin{split}
&\left\|
B\text{\rm diag}_{m_1}\left(D\text{\rm diag}_{m_1}\left(\cdots
D\text{\rm diag}_{m_1}\left(\widehat\Delta_{{\bf T}_2}\right)
 {\bf T}_1^*\cdots\right){\bf T}_1^*\right){\bf T}_1^*h\right\|\\
 &\qquad
 \leq \|B\|\left(\sum_{{\alpha_1,\ldots, \alpha_{m+2}\in \FF_{n_1}^+}\atop { 1\leq |\alpha_i|\leq k_1}}\|\sqrt{a_{\alpha_1}\cdots a_{\alpha_{m+2}}}(T_{1,\alpha_{m+2}}\cdots T_{1,\alpha_1})^*h\|^2\right)^{1/2}=\left< \Phi_{f,{\bf T}_1}^{m+2}(I)h, h\right>
 \end{split}
 \end{equation*}
 for any $h\in \cH$.
 Since ${\bf T}_1$ is pure in $\cD_f(\cH)$,  we have
 $\lim_{m\to \infty}\Phi_{f,{\bf T}_1}^{m+2}(I)h=0$ for any $h\in \cH$.
Consequently, relation \eqref{rel} implies
\begin{equation*}
  \begin{split}
  \text{\rm diag}_{m_2}(\Delta_{{\bf T}_1'}) {\bf X}_2^*h
=A\Delta_{{\bf T}_1}h
+B\sum_{p=0}^\infty \text{\rm diag}_{m_1}\left(D\text{\rm diag}_{m_1}\left(\cdots
D\text{\rm diag}_{m_1}\left(C\Delta_{{\bf T}_1}\right) {\bf X}_1^*\cdots\right){\bf X}_1^*\right){\bf X}_1^*h,
 \end{split}
  \end{equation*}
for any $h\in \cH$, where $\text{\rm diag}_{m_1}$ appears $p+1$ times in the general term of  the series.
The proof is complete.
\end{proof}

 We recall   (\cite{Po-analytic}, \cite{Po-domains})
       a few facts
       concerning multi-analytic   operators on Fock
      spaces.
         We say that
       a bounded linear
        operator
      $M$ acting from $F^2(H_n)\otimes \cK$ to $ F^2(H_n)\otimes \cK'$ is
       multi-analytic with respect to the universal model ${\bf W}:=(W_1,\ldots, W_n)$
      if
      $
      M(W_i\otimes I_\cK)= (W_i\otimes I_{\cK'}) M\quad
      \text{\rm for any }\ i=1,\dots, n.
      $
       We can associate with $M$ a unique formal Fourier expansion
      $  \sum_{\alpha \in \FF_n^+}
      \Lambda_\alpha \otimes \theta_{(\alpha)}$
where $\theta_{(\alpha)}\in B(\cK, \cK')$.
       We  know  that
        $M =\text{\rm SOT-}\lim_{r\to 1}\sum_{k=0}^\infty
      \sum_{|\alpha|=k}
         r^{|\alpha|} \Lambda_\alpha\otimes \theta_{(\alpha)}
         $
         where, for each $r\in [0,1)$, the series converges in the uniform norm.
      Moreover, the set of  all multi-analytic operators in
      $B(F^2(H_n)\otimes \cK,
      F^2(H_n)\otimes \cK')$  coincides  with
      $\cR_n^\infty(\cD_f)\bar \otimes B(\cK,\cK')$,
      the WOT-closed operator space generated by the spatial tensor
      product.

Let $\cH$, $\cH'$, and $\cE$ be Hilbert spaces and
consider
$$ U=\left[ \begin{matrix} A&B\\C&D \end{matrix}\right]:\begin{matrix}\cH\\\oplus \\ \bigoplus\limits_{{\alpha\in \FF_{n_1}^+},{ 1\leq|\alpha|\leq k_1}}\cE\end{matrix}\to  \begin{matrix} \bigoplus\limits_{{\beta\in \FF_{n_2}^+},{ 1\leq|\beta|\leq k_2}}\cH'\\\oplus \\\cE\end{matrix}
$$
to be a unitary operator. Setting
$D=[D_{(\alpha)}: \quad \alpha\in \FF_{n_1}^+, 1\leq |\alpha|\leq 1]: \bigoplus\limits_{{\alpha\in \FF_{n_1}^+},{ 1\leq|\alpha|\leq k_1}}\cE\to \cE,
$
we associate with $U^*$  and any $r\in [0,1)$ the operator
$\varphi_{U^*}(r{\bf \Lambda}_1)$ defined by
\begin{equation*}
\begin{split}
\varphi_{U^*}(r{\bf \Lambda}_1):= I_{F^2(H_{n_1})}\otimes A^*&+\left(I_{F^2(H_{n_1})}\otimes C^*\right) \left(I_{F^2(H_{n_1})\otimes \cE}-\sum\limits_{{\alpha\in \FF_{n_1}^+}\atop{ 1\leq |\alpha|\leq k_1}}
r^{|\alpha|}\sqrt{a_{ \alpha}} \Lambda_{1,\tilde\alpha}\otimes D_{(\alpha)}^*\right)^{-1}\\
& \times \left[\sqrt{a_{  \alpha}} \Lambda_{1,\tilde\alpha}\otimes I_\cH:\ \alpha\in \FF_{n_1}^+, 1\leq |\alpha|\leq k_1\right]\left(I_{F^2(H_{n_1})}\otimes B^*\right),
\end{split}
\end{equation*}
where ${\bf \Lambda}_1:=[\Lambda_{1,1},\ldots, \Lambda_{1,n_1}]$ is the tuple of weighted right creation operators on $F^2(H_{n_1})$ associated with the regular domain $\cD_f$.
In what follows, we use the notations: ${\bf A}:=I_{F^2(H_{n_1})}\otimes A$, ${\bf B}:=I_{F^2(H_{n_1})}\otimes B$,
 ${\bf C}:=I_{F^2(H_{n_1})}\otimes C$, ${\bf D}:=I_{F^2(H_{n_1})}\otimes D$,
    and
 ${\bf \Gamma}(r):=\left[\sqrt{a_{  \alpha}} r^{|\alpha|} \Lambda_{1,\tilde\alpha}\otimes I_\cE:\ \alpha\in \FF_{n_1}^+, 1\leq |\alpha|\leq k_1\right]$.

\begin{lemma} \label{strong-limit} The strong operator topology limit
$\varphi_{U^*}({\bf \Lambda}_1):=\text{\rm SOT-}\lim_{r\to 1}\varphi_{U^*}(r{\bf \Lambda}_1)
$
 exists and defines a contractive  multi-analytic  operator with respect to ${\bf W}_1$ having the  row matrix representation
 $$\varphi_{U^*}({\bf \Lambda}_1)=[\varphi_{(\beta)}({\bf \Lambda}_1): \ \beta\in \FF_{n_2}^+, 1\leq |\beta|\leq k_2],
 $$
 with  $\varphi_{(\beta)}({\bf \Lambda}_1)\in \cR_{n_1}^\infty(\cD_{f})\bar\otimes B\left( \cH',\cH\right)$, where $\cR_{n_1}^\infty(\cD_{f}) $ is the noncommutative Hardy  algebra generated by the  weighted right creation operators $\Lambda_{1,1},\ldots, \Lambda_{1,n_1}$ and the identity.

\end{lemma}

\begin{proof} Since ${\bf D}$ and ${\bf \Gamma}(r)$ are contractions, we have $\left\|\sum\limits_{{\alpha\in \FF_{n_1}^+},{ 1\leq |\alpha|\leq k_1}}
r^{|\alpha|}\sqrt{a_{  \alpha}} \Lambda_{1,\tilde\alpha}\otimes D_{(\alpha)}^*\right\|\leq r<1$. Consequently, the operator $\varphi_{U^*}(r{\bf \Lambda}_1)$ makes sense and
$\varphi_{U^*}(r{\bf \Lambda}_1)=[\varphi_{(\beta)}(r{\bf \Lambda}_1): \ \beta\in \FF_{n_2}^+, 1\leq |\beta|\leq k_2 ]$ with  $\varphi_{(\beta)}(r{\bf \Lambda}_1)\in \cR_{n_1} (\cD_{f})\bar\otimes B\left( \cH',\cH\right)$, where $\cR_{n_1} (\cD_{f})$ is the noncommutative disk algebra generated by $\Lambda_{1,1},\ldots, \Lambda_{1,n_1}$ and the identity.
Consequently,
\begin{equation}\label{fi-rep}
\varphi_{(\beta)}(r{\bf \Lambda}_1)=\sum_{k=0}^\infty \sum_{\gamma\in \FF_{n_1}^+, |\gamma|=k} r^{|\gamma|} {\bf \Lambda}_{1,\gamma}\otimes \Theta_{(\gamma)}^{(\beta)}
\end{equation}
for some operators $\Theta_{(\gamma)}^{(\beta)}\in B(\cH',\cH)$, where the convergence is in the operator norm topology.
On the other hand, since $ \left[ \begin{matrix} {\bf A}^*&{\bf C}^*\\{\bf B}^*&{\bf  D}^* \end{matrix}\right]
$
is a unitary operator, standard calculations (see e.g. \cite{Po-varieties})    show that
\begin{equation*}
\begin{split}
I&-\varphi_{U^*}(r{\bf \Lambda}_1)\varphi_{U^*}(r{\bf \Lambda}_1)^*\\
&=  {\bf C}^*(I-{\bf  \Gamma}(r) {\bf D}^*)^{-1}\left[\left( I-\sum_{\alpha\in \FF_{n_1}^+, 1\leq |\alpha|\leq k_1} r^{2|\alpha|}a_{  \alpha} \Lambda_{1,\tilde\alpha} \Lambda_{1,\tilde\alpha}^*\right)\otimes I\right] (I-  {\bf D}{\bf \Gamma}(r)^*)^{-1}{\bf C}.
\end{split}
\end{equation*}
This    shows that $\varphi_{U^*}(r{\bf \Lambda}_1)$ is a contraction for any $r\in [0,1)$ having the row matrix representation    $\varphi_{U^*}(r{\bf \Lambda}_1)=[\varphi_{(\beta)}(r{\bf \Lambda}_1): \ \beta\in \FF_{n_2}^+, 1\leq |\beta|\leq k_2 ]$ with  $\varphi_{(\beta)}(r{\bf \Lambda}_1)\in \cR_{n_1} (\cD_{f})\bar\otimes B\left( \cH',\cH\right)$. Since $\varphi_{(\beta)}(r{\bf \Lambda}_1)$
 has the  Fourier representation \eqref{fi-rep}, one can see
 that
  $\varphi_{(\beta)}({\bf \Lambda}_1):=\text{\rm SOT-}\lim_{r\to 1}\varphi_{(\beta)}(r{\bf \Lambda}_1)$ exists,
  $\varphi_{(\beta)}({\bf \Lambda}_1)\in \cR_{n_1}^\infty\bar\otimes B\left(\cH',\cH\right)$,  and $\|\varphi_{U^*}({\bf \Lambda}_1)\|\leq 1$.
The proof is complete.
\end{proof}

 We recall that $\cI({\bf T}_1,{\bf T}_1')$ is  the set of all tuples  ${\bf T}_2:=(T_{2,1},\ldots, T_{2,n_2})$, with $T_{2,j}:\cH'\to \cH$,  such that ${\bf T}_2\in \cD_{g}(\cH', \cH)$    intertwines ${\bf T}_1$ with ${\bf T}_1'$, i.e.
   $
   T_{2,j}T_{1,i}'=T_{1,i}T_{2,j}
   $
   for any $i\in \{1,\ldots, n_1\}$ and $j\in \{1,\ldots, n_2\}$.
   The main result of this section is the following  intertwining  dilation theorem for  the elements of $\cI({\bf T}_1,{\bf T}_1')$.

\begin{theorem} \label{dil}  Let
  ${\bf T}_1:=(T_{1,1},\ldots, T_{1,n_1})\in \cD_f(\cH)$ and  ${\bf T}_1':=(T_{1,1}',\ldots, T_{1,n_1}')\in \cD_f(\cH')$, and  let ${\bf T}_2:=(T_{2,1},\ldots, T_{2,n_2})\in \cD_g(\cH', \cH)$ be  such that ${\bf T}_2\in \cI({\bf T}_1,{\bf T}_1')$. Let ${\bf W}_1:=(W_{1,1},\ldots, W_{1,n_1})$ and ${\bf \Lambda}_1:=(\Lambda_{1,1},\ldots, \Lambda_{1, n_1})$ be the weighted creation operators associated with the noncommutative domain $\cD_f$.
If
$\varphi_{U^*}({\bf \Lambda}_1)=[\varphi_{(\beta)}({\bf \Lambda}_1): \ \beta\in \FF_{n_2}^+, 1\leq |\beta|\leq k_2 ]$ is the contractive multi-analytic operator associated with $U\in \cU_{\bf T}^\cK$ and
   ${\bf T}_1$ is a  pure element of the noncommutative regular domain $\cD_f(\cH)$,   then the following relations hold:
$$
K_{f,{\bf T}_1'} T_{2,\beta}^* =\frac{1}{\sqrt{c_{\beta}}}\varphi_{(\beta)} ({\bf \Lambda}_1)^*K_{f,{\bf T}_1}, \qquad \beta\in \FF_{n_2}^+, 1\leq |\beta|\leq k_2,
$$
and
$$
K_{f,{\bf T}_1}T_{1,i}^*=
\left(W_{1,i}^*\otimes I_{\cD_{{\bf T}_1}}\right)  K_{f,{\bf T}_1},\quad
K_{f,{\bf T}_1'}(T_{1,i}')^*=
\left(W_{1,i}^*\otimes I_{\cD_{{\bf T}_1'}}\right)  K_{f,{\bf T}_1'}, \qquad i\in \{1,\ldots, n_1\},
$$
 where  $K_{f,{\bf T}_1}$ and $K_{f,{\bf T}_1'}$ are the   noncommutative Poisson kernels associated with ${\bf T}_1$ and ${\bf T}_1'$, respectively.
\end{theorem}

\begin{proof} Fix $U=\left[ \begin{matrix} A&B\\C&D \end{matrix}\right]\in \cU_{\bf T}^\cK$ and set
$$D:=[D_{(\alpha)}: \quad \alpha\in \FF_{n_1}^+, 1\leq |\alpha|\leq 1]: \bigoplus_{\alpha\in \FF_{n_1}^+, 1\leq |\alpha|\leq k_1}(\cD_{{\bf T}_2}\oplus \cK)\to \cD_{{\bf T}_2}\oplus \cK.
$$
In what follows, we use the notations: ${\bf A}:=I_{F^2(H_{n_1})}\otimes A$, ${\bf B}:=I_{F^2(H_{n_1})}\otimes B$,
 ${\bf C}:=I_{F^2(H_{n_1})}\otimes C$,
 \begin{equation*}
 \begin{split}
 {\bf Q}&:=\sum_{\alpha\in \FF_{n_1}^+, 1\leq |\alpha|\leq k_1} a_{ \alpha} \Lambda_{1,{\tilde\alpha}}\otimes D_{(\alpha)}^* \text{ and }
{\bf \Gamma}:=\left[\sqrt{a_{ \sigma}}\Lambda_{1,{\tilde\sigma}}\otimes I_{\cD_{{\bf T}_2}\oplus \cK}: \sigma\in \FF_{n_1}^+, 1\leq |\sigma|\leq k_1\right].
\end{split}
\end{equation*}
As in Lemma \ref{le1}, an induction argument  over $q$ shows that
\begin{equation}\label{diag-q}
\begin{split}
\text{\rm diag}_{m_1}\left(\text{\rm diag}_{m_1}\cdots\left(
\text{\rm diag}_{m_1}\left(C^*\right)
D^*\right)\cdots D^*\right) =
 \text{\rm diag}_{m_1}\left(\left[\begin{matrix}C^*(D_{ (\gamma_1)}\cdots D_{ (\gamma_q)})^*\\
 \vdots\\
 \gamma_i\in \FF_{n_1}^+, 1\leq |\gamma_i|\leq k_1\end{matrix}\right]\right),
\end{split}
\end{equation}
where $\text{\rm diag}_{m_1}$ appears  $q+1$ times on the left-hand side of the equality.

We associate with $U^*$   the  multi-analytic operator $\varphi_{U^*}({\bf \Lambda}_1)\in \cR_{n_1}^\infty(\cD_f)\bar\otimes B\left( \bigoplus_{\beta\in \FF_{n_2}^+, 1\leq |\beta|\leq k_2}\cD_{{\bf T}_1'}, \cD_{{\bf T}_1}\right)$,  as in Lemma \ref{strong-limit}, in the particular case when $\cH=\cD_{{\bf T}_1}$, $\cH'=\cD_{{\bf T}_1'}$, and $\cE=\cD_{{\bf T}_2}\oplus \cK$.
Note that, for each $y\in  \bigoplus_{\beta\in \FF_{n_2}^+, 1\leq |\beta|\leq k_2}\cD_{{\bf T}_1'}$ and $\alpha\in \FF_{n_1}^+$ with $|\alpha|=n$, we have

$$
\varphi_{U^*}({\bf \Lambda}_1)(e_\alpha\otimes y)=e_\alpha\otimes A^*y+
\sum_{q=0}^\infty {\bf C}^*{\bf Q}^q {\bf \Gamma}{\bf B}^*(e_\alpha\otimes y)
$$
and ${\bf C}^*{\bf Q}^q {\bf \Gamma}{\bf B}^*(e_\alpha\otimes y)$ is in the closed linear span of all the vectors $e_{\alpha\alpha_1\cdots \alpha_{q+1}}\otimes z$, where $\alpha_1,\ldots, \alpha_{q+1}\in \FF_{n_1}^+, 1\leq|\alpha_i|\leq k_1$ and $z\in \cD_{{\bf T}_1}$. Consequently, using the noncommutative Poisson kernel $K_{{\bf T}_1}$,  we  deduce that
\begin{equation}\label{three}
\begin{split}
&\left<\varphi_{U^*}({\bf \Lambda}_1)^*K_{f,{\bf T}_1}h, e_\alpha\otimes y\right>\\
&=\left< \sum_{k=0}^\infty \sum_{\beta\in \FF_{n_1}^+,|\beta|=k}
 \sqrt{b_\beta} e_\beta\otimes \Delta_{{\bf T}_1}T_{1,\beta}^*h, \varphi_{U^*}({\bf \Lambda}_1)(e_\alpha\otimes y)\right>\\
&=
\left<\sqrt{b_\alpha}\Delta_{{\bf T}_1}T_{1,\alpha}^*h, A^*y\right>
+\sum_{q=0}^\infty \left< \sum_{{\gamma\in \FF_{n_1}^+} }\sqrt{b_{\alpha \gamma}}e_{\alpha\gamma}\otimes \Delta_{{\bf T}_1}T_{1,\gamma}^* T_{1,\alpha}^*h, {\bf C}^*{\bf Q}^q {\bf \Gamma}{\bf B}^*(e_\alpha\otimes y)
\right>,
\end{split}
\end{equation}
for any $h\in \cH$, $y\in  \bigoplus_{\beta\in \FF_{n_2}^+, 1\leq |\beta|\leq k_2}\cD_{{\bf T}_1'}$, and $\alpha\in \FF_{n_1}^+$ with $|\alpha|=n$.
Setting
$$
B:=[B_{(\alpha)}: \quad \alpha\in \FF_{n_1}^+, 1\leq |\alpha|\leq 1]:\bigoplus_{\alpha\in \FF_{n_1}^+, 1\leq |\alpha|\leq k_1}(\cD_{{\bf T}_2}\oplus \cK)\to \bigoplus_{\beta\in \FF_{n_2}^+, 1\leq |\beta|\leq k_2}\cD_{{\bf T}_1'},
$$
we obtain
\begin{equation*}
\begin{split}
&{\bf C}^*{\bf Q}^q {\bf \Gamma}{\bf B}^*(e_\alpha\otimes y)\\
&=
(I_{F^2(H_{n_1})}\otimes C^*)\left(\sum_{{\alpha_1,\ldots, \alpha_q\in \FF_{n_1}^+}\atop{1\leq |\alpha_i|\leq k_1}}\sqrt{a_{\alpha_1}\cdots a_{\alpha_q}}\Lambda_{1,{\tilde\alpha_1}}\cdots \Lambda_{1,{\tilde\alpha_q}}\otimes
 D^*_{(\alpha_1)}\cdots D^*_{(\alpha_q)}\right)\\
&
\left[\sqrt{a_{ \sigma}}\Lambda_{1,{\tilde\sigma}}\otimes I_{\cD_{{\bf T}_2}\oplus \cK}: \sigma\in \FF_{n_1}^+, 1\leq |\sigma|\leq k_1\right]
\left[\begin{matrix} I_{F^2(H_{n_1})}\otimes B_{(\sigma)}^*\\
\vdots\\\sigma\in \FF_{n_1}^+, 1\leq |\sigma|\leq k_1\end{matrix}\right](e_\alpha\otimes y)\\
&=
 \sum_{{\sigma\in \FF_{n_1}^+}\atop {1\leq |\sigma|\leq k_1}}
 \left(\sum_{{\alpha_1,\ldots, \alpha_q\in \FF_{n_1}^+}\atop{1\leq |\alpha_i|\leq k_1}}\sqrt{a_{\alpha_1}\cdots a_{\alpha_q}a_{\sigma}}\Lambda_{1,{\tilde\alpha_1}}\cdots \Lambda_{1,{\tilde\alpha_q}}\Lambda_{1,{\tilde \sigma}}e_\alpha\otimes
 C^*D^*_{(\alpha_1)}\cdots D^*_{(\alpha_q)}B_{(\sigma)}^*y\right)
\\
&=
 \sum_{{\sigma\in \FF_{n_1}^+}\atop {1\leq |\sigma|\leq k_1}}
 \left(\sum_{{\alpha_1,\ldots, \alpha_q\in \FF_{n_1}^+}\atop{1\leq |\alpha_i|\leq k_1}}\sqrt{a_{\alpha_1}\cdots a_{\alpha_q}a_{\sigma}}
 \frac{\sqrt{b_\alpha}}{\sqrt{b_{\alpha\sigma\alpha_q\cdots \alpha_1}}}
 e_{\alpha\sigma\alpha_q\cdots \alpha_1}
 C^*D^*_{(\alpha_1)}\cdots D^*_{(\alpha_q)}B_{(\sigma)}^*y\right),
\end{split}
\end{equation*}
where $\widetilde\sigma$ is the reverse of $\sigma\in \FF_{n_1}^+$.
Consequently, using relation \eqref{diag-q}, we deduce that
\begin{equation*}
\begin{split}
&\sum_{q=0}^\infty \left< \sum_{{\gamma\in \FF_{n_1}^+} }\sqrt{b_{\alpha \gamma}}e_{\alpha\gamma}\otimes \Delta_{{\bf T}_1}T_{1,\gamma}^* T_{1,\alpha}^*h, {\bf C}^*{\bf Q}^q {\bf \Gamma}{\bf B}^*(e_\alpha\otimes y)
\right>\\
&=
\sum_{q=0}^\infty \left< \sum_{{\gamma\in \FF_{n_1}^+} }\sqrt{b_{\alpha \gamma}}e_{\alpha\gamma}\otimes \Delta_{{\bf T}_1}T_{1,\gamma}^* T_{1,\alpha}^*h,\right.\\
 &\qquad \qquad \left.\sum_{{\sigma\in \FF_{n_1}^+}\atop {1\leq |\sigma|\leq k_1}}
 \left(\sum_{{\alpha_1,\ldots, \alpha_q\in \FF_{n_1}^+}\atop{1\leq |\alpha_i|\leq k_1}}\sqrt{a_{\alpha_1}\cdots a_{\alpha_q}a_{\sigma}}
 \frac{\sqrt{b_\alpha}}{\sqrt{b_{\alpha\sigma\alpha_q\cdots \alpha_1}}}
 e_{\alpha\sigma\alpha_q\cdots \alpha_1}
 C^*D^*_{(\alpha_1)}\cdots D^*_{(\alpha_q)}B_{(\sigma)}^*y\right)
\right>\\
&=
\sum_{q=0}^\infty
\sum_{{\sigma\in \FF_{n_1}^+}\atop {1\leq |\sigma|\leq k_1}}
 \sum_{{\alpha_1,\ldots, \alpha_q\in \FF_{n_1}^+}\atop{1\leq |\alpha_i|\leq k_1}}
 \left<
 \sqrt{b_{\alpha\sigma\alpha_q\cdots \alpha_1}}e_{\alpha\sigma\alpha_q\cdots \alpha_1}\otimes \Delta_{{\bf T}_1}T_{1,\sigma\alpha_q\cdots \alpha_1}^* T_{1,\alpha}^*h,\right.\\
 &\qquad \qquad\qquad \qquad\left.\sqrt{a_{\alpha_1}\cdots a_{\alpha_q}a_{\sigma}}
 \frac{\sqrt{b_\alpha}}{\sqrt{b_{\alpha\sigma\alpha_q\cdots \alpha_1}}}
 e_{\alpha\sigma\alpha_q\cdots \alpha_1}
 C^*D^*_{(\alpha_1)}\cdots D^*_{(\alpha_q)}B_{(\sigma)}^*y
 \right>\\
 &=\sum_{q=0}^\infty
\sum_{{\sigma\in \FF_{n_1}^+}\atop {1\leq |\sigma|\leq k_1}}
 \sum_{{\alpha_1,\ldots, \alpha_q\in \FF_{n_1}^+}\atop{1\leq |\alpha_i|\leq k_1}}
 \left<\sqrt{b_\alpha}\sqrt{a_{\alpha_1}\cdots a_{\alpha_q}a_{\sigma}}\Delta_{{\bf T}_1}T_{1,\sigma\alpha_q\cdots \alpha_1}^* T_{1,\alpha}^*h, C^*D^*_{(\alpha_1)}\cdots D^*_{(\alpha_q)}B_{(\sigma)}^*y
 \right>\\
 &=
 \sum_{q=0}^\infty
 \left<
 \left[ \begin{matrix}
 \left[ \begin{matrix} \sqrt{a_{\alpha_1}\cdots a_{\alpha_q}a_{\sigma}}\Delta_{{\bf T}_1}T_{1,\sigma\alpha_q\cdots \alpha_1}^*\\
 \vdots\\
 {\alpha_1,\ldots, \alpha_q\in \FF_{n_1}^+}\atop{1\leq |\alpha_i|\leq k_1}
 \end{matrix}\right]\\
 \vdots\\
 \sigma\in \FF_{n_1}^+, 1\leq |\sigma|\leq k_1
 \end{matrix}\right]\sqrt{b_\alpha} T_{1,\alpha}^*h,
 \left[ \begin{matrix}
 \left[ \begin{matrix} C^*D^*_{(\alpha_1)}\cdots D^*_{(\alpha_q)}B_{(\sigma)}^*y \\
 \vdots\\
 {\alpha_1,\ldots, \alpha_q\in \FF_{n_1}^+}\atop{1\leq |\alpha_i|\leq k_1}
 \end{matrix}\right]\\
 \vdots\\
 \sigma\in \FF_{n_1}^+, 1\leq |\sigma|\leq k_1
 \end{matrix}\right]
 \right>\\
 &=\sum_{q=0}^\infty\left<B \text{\rm diag}_{m_1}\left(D\text{\rm diag}_{m_1}\left(\cdots
D\text{\rm diag}_{m_1}\left(C\Delta_{{\bf T}_1}\right) {\bf X}_1^*\cdots\right){\bf X}_1^*\right){\bf X}_1^*\left(\sqrt{b_\alpha} T_{1,\alpha}^*\right)h, y\right>.
\end{split}
\end{equation*}
Hence and using relation \eqref{three}, we obtain
\begin{equation*}
\begin{split}
&\left<\varphi_{U^*}({\bf \Lambda}_1)^*K_{f,{\bf T}_1}h, e_\alpha\otimes y\right>\\
&=\left<\sqrt{b_\alpha}\Delta_{{\bf T}_1}T_{1,\alpha}^*h, A^*y\right>
+\sum_{q=0}^\infty\left<B \text{\rm diag}_{m_1}\left(D\text{\rm diag}_{m_1}\left(\cdots
D\text{\rm diag}_{m_1}\left(C\Delta_{{\bf T}_1}\right) {\bf X}_1^*\cdots\right){\bf X}_1^*\right){\bf X}_1^*\left(\sqrt{b_\alpha} T_{1,\alpha}^*\right)h, y\right>.
\end{split}
\end{equation*}
Now, using Lemma \ref{series}, we deduce that
\begin{equation*}
  \begin{split}
  \left<\varphi_{U^*}({\bf \Lambda}_1)^*K_{f,{\bf T}_1}h, e_\alpha\otimes y\right>=
  \left<\text{\rm diag}_{m_2}(\Delta_{{\bf T}_1'})\left[\begin{matrix}
  \sqrt{c_\beta}T_{2,\beta}^*\\
  \vdots\\
  \beta\in \FF_{n_2}^+, 1\leq |\beta|\leq k_2\end{matrix}
  \right] \left(\sqrt{b_\alpha} T_{1,\alpha}^*\right)h, y\right>
 \end{split}
  \end{equation*}
for any $h\in \cH$.
Hence, using the definition of the noncommutative Poisson kernel and the fact that ${\bf T}_2\in \cI({\bf T}_1,{\bf T}_1')$, we deduce that, for any $\beta\in \FF_{n_2}^+$ with $1\leq |\beta|\leq k_1$,  $h\in \cH$, and $z\in \cD_{{\bf T}_1'}$,
\begin{equation*}
\begin{split}
\left<K_{f,{\bf T}_1'}\sqrt{c_\beta} T_{2,\beta}^*h, e_\alpha\otimes z\right>
&=\left<
\sum_{k=0}^\infty \sum_{\sigma\in \FF_{n_1}^+|\sigma|=k}  e_\sigma\otimes
\sqrt{b_\sigma}\Delta_{{\bf T}_1'} (T_{1,\sigma}')^*\sqrt{c_\beta} T_{2,\beta}^*h,e_\alpha\otimes z\right>\\
&=
\left<\sqrt{b_\alpha}\Delta_{{\bf T}_1'} (T_{1,\alpha}')^*\sqrt{c_\beta}T_{2,\beta}^*h,z\right>
=
\left<\sqrt{b_\alpha}\Delta_{{\bf T}_1'} \sqrt{c_\beta}T_{2,\beta}^*T_{1,\alpha}^*h,z\right>\\
&=
\left<\varphi_{U^*}({\bf R})^*K_{{\bf T}_1}h,   e_\alpha\otimes y \right>=
\left<\varphi_{(\beta)}({\bf R})^*K_{{\bf T}_1}h, e_\alpha\otimes z\right>,
\end{split}
\end{equation*}
where $y=\bigoplus_{{\gamma\in \FF_{n_2}^+}\atop{ 1\leq|\gamma|\leq k_2} } y_{(\gamma)}$
with $y_{(\gamma)}=0$ if $\gamma\neq \beta$ and $y_{(\beta)}=z$.
Consequently,
$$
K_{f,{\bf T}_1'} T_{2,\beta}^* =\frac{1}{\sqrt{c_{\beta}}}\varphi_{(\beta)} ({\bf \Lambda}_1)^*K_{f,{\bf T}_1}, \qquad \beta\in \FF_{n_2}^+, 1\leq |\beta|\leq k_2.
$$
   The last two relations in the theorem are due to relation \eqref{ker-inter} applied to ${\bf T}_1\in \cD_f(\cH)$ and   ${\bf T}_1'\in \cD_f(\cH')$, respectively.
The proof is complete.
\end{proof}

As a consequence of Theorem \ref{dil}, we  obtain a new proof for the commutant lifting theorem for the  pure elements of the noncommutative domain $\cD_f$ (see \cite{Po-domains}) as well as  a constructive method to obtain the lifting.

\begin{theorem} \label{CLT}
Let ${\bf T}_1:=(T_{1, 1},\ldots, T_{1,n_1})\in \cD_f(\cH)$ and  ${\bf T}_1'=(T_{1,1}',\ldots, T_{1,n_1}')\in \cD_f(\cH')$ be pure tuples of operators and let
 ${\bf W}_1:=[W_{1,1},\ldots, W_{1,n_1}]$    be the universal model associated with the noncommutative domain $\cD_f$.
 If $A:\cH'\to \cH$  is  an operator such that
$$AT_{1,i}'=T_{1,i}A, \qquad i\in\{1,\ldots, n_1\},
 $$
 then there is an operator $D:F^2(H_{n_1})\otimes \cD_{{\bf T}_1'}\to F^2(H_{n_1})\otimes \cD_{{\bf T}_1}$ such that
$$
D(W_{1,i}\otimes I_{\cD_{{\bf T}'_1}})=(W_{1,i}\otimes I_{\cD_{{\bf T}_1}})D,  \qquad
i\in \{1,\ldots, n_1\},
$$
 $D^*|_\cH=A^*$, and $\|B\|=\|A\|$, where $\cH$ and $\cH'$ are identified with co-invariant subspaces of $\{W_{1,i}\otimes I_{\cD_{\bf T}}\}_{i=1}^{n_1}$ and $\{W_{1,i}\otimes I_{\cD_{{\bf T}'}}\}_{i=1}^{n_1}$, respectively.
\end{theorem}
\begin{proof}

Without loss of generality, we can assume that $\|A\|=1$.
Since $A\in \cI({\bf T}_1', {\bf T}_1)$, we can apply Theorem \ref{dil} in the particular case when $n_2=1$, ${\bf T}_2:=A$, and $g=X$. Consequently, there is a contractive multi-analytic operator $\varphi({\bf \Lambda}_1)\in \cR_{n_1}^\infty(\cD_f)\bar \otimes B(\cD_{{\bf T}_1'},\cD_{{\bf T}_1})$ such that
$K_{f,{\bf T}_1'} A^*=\varphi({\bf \Lambda}_1)^* K_{f,{\bf T}_1}$.
 Since ${\bf T}_1$ and ${\bf T}_1'$ are pure elements, the noncommutative Poisson kernels $K_{f,{\bf T}_1}$ and $K_{f,{\bf T}_1'}$ are isometries.
 Under the identifications of $\cH$  and $\cH'$ with $K_{f,{\bf T}_1}\cH$ and $K_{f,{\bf T}'_1}\cH'$, respectively,  we have
$A^*=\varphi({\bf \Lambda}_1)^*|_\cH$.
Since $1=\|A^*\|\leq \|\varphi({\bf \Lambda}_1)^*\|\leq 1$, we deduce that $\|A\|=\|\varphi({\bf \Lambda}_1)\|$.
Since $D:=\varphi({\bf \Lambda}_1)$ intertwines $W_{1,i}\otimes I_{\cD_{{\bf T}_1'}}$ with
$W_{1,i}\otimes I_{\cD_{{\bf T}_1}}$ for each $i\in \{1,\ldots, n_1\}$, the proof is complete.
\end{proof}

More  applications of Theorem \ref{dil} will be considered in the next sections.

\bigskip

\section{Noncommutative varieties, dilations,  and Schur representations}

In this section, we obtain an  intertwining dilation theorem on noncommutative varieties in regular domains  and   a  Schur  type representation   for the  unit ball of $\cR_{n}^\infty(\cV_J)\bar \otimes B(\cH', \cH)$.

First, we  recall from \cite{Po-varieties} and \cite{Po-domains} basic facts concerning  noncommutative  varieties generated by $WOT$-closed  two-sided ideals of  the Hardy algebra $F^\infty_n(\cD_f)$, their universal models,  and the Hardy algebras they generate.
Let
  $J$ be a $WOT$-closed  two-sided ideal of $F^\infty_n(\cD_f)$ such that $J\neq F^\infty_n(\cD_f)$. We introduce the noncommutative variety $\cV_{J}(\cH)$ to be the set of all   pure $n$-tuples
${\bf T}:=(T_1,\ldots, T_n)\in \cD_f(\cH)$    with the property that
  $$
\varphi(T_1,\ldots, T_n)=0\quad \text{for any } \ \varphi\in J,
$$
where $\varphi(T_1,\ldots, T_n)$ is defined   using  the $F_n^\infty(\cD_f)$-functional calculus for pure elements in $\cD_f(\cH)$.
Define the subspaces of
$F^2(H_n)$ by
$$
\cM_J:=\overline{JF^2(H_n)}\quad \text{and}\quad
\cN_J:=F^2(H_n)\ominus \cM_J.
$$
The subspace $\cN_J$  is  invariant under
the operators
 $W_1^*,\ldots, W_n^*$ and  $\Lambda_1^*,\ldots, \Lambda_n^*$, and    $\cN_J\neq 0$ if and only if $ J\neq
F_n^\infty(\cD_f)$.
Define the {\it constrained  weighted left} (resp.~{\it right}) {\it
creation operators} associated with the noncommutative variety
$\cV_{J}$ by setting
$$B_i:=P_{\cN_J} W_i|_{\cN_J} \quad \text{and}\quad C_i:=P_{\cN_J} \Lambda_i|_{\cN_J},\quad i=1,\ldots, n.
$$
We remark that
 ${\bf B}:=(B_1,\ldots, B_n)$ is in  $\cV_{J}(\cN_J)$ and  plays the role of
universal model for  the noncommutative variety $\cV_{J}$. We will refer to the $n$-tuples ${\bf B}:=(B_1,\ldots, B_n)$ and ${\bf C}:=(C_1,\ldots, C_n)$ as the constrained weighted creation operators associated with $\cV_J$. Note that if $J=\{0\}$, then $\cN_{\{0\}}=F^2(H_n)$ and  $\cV_{\{0\}}(\cH)$ is the set of all pure elements of $\cD_f(\cH)$.

Let $F_n^\infty(\cV_{J})$ be the $WOT$-closed algebra generated by
$B_1,\ldots, B_n$ and the identity and let $R_n^\infty(\cV_{J})$ be the $WOT$-closed
algebra generated by $C_1,\ldots, C_n$ and the identity. We proved  in \cite{Po-domains} that
\begin{equation*}
F_n^\infty(\cV_{J})^\prime=R_n^\infty(\cV_{J})\ \text{ and } \
R_n^\infty(\cV_{J})^\prime=F_n^\infty(\cV_{J}),
\end{equation*}
where $^\prime$ stands for the commutant. An operator $M\in
B(\cN_J\otimes \cK,\cN_J\otimes \cK')$ is called multi-analytic with
respect to the  universal model  ${\bf B}:=(B_1,\ldots, B_n)$ if
$
M(B_i\otimes I_{\cK})=(B_i\otimes I_{\cK'})M$ for $i=1,\ldots, n.
$
  We recall
that the set of all multi-analytic operators with respect to
 ${\bf B}$ coincides  with
$$
R_n^\infty(\cV_{J})\bar\otimes B(\cK,\cK')=P_{\cN_J\otimes
\cK'}[R_n^\infty(\cD_f)\bar\otimes B(\cK,\cK')]|_{\cN_J\otimes \cK}.
$$
A similar result holds for  the  Hardy algebra $F_n^\infty(\cV_{J})$.
Given a noncommutative variety $\cV_J(\cH)$ and
${\bf T}\in \cV_J(\cH)$, we define
the {\it constrained Poisson kernel} $K_{J,{\bf T}}:\cH\to \cN_J\otimes \cD_{\bf T}$  by
$$K_{J,{\bf T}}:=(P_{\cN_J}\otimes I_{\cD_{\bf T}})K_{J,{\bf T}}.
$$
We recall that $K_{J,{\bf T}}$
is an isometry and satisfies the relation
\begin{equation} \label{KJK}
K_{J,{\bf T}}T_\alpha^*=(B_\alpha^*\otimes I_{\cD_{\bf T}})K_{J,{\bf T}},\quad  \alpha\in \FF_n^+.
\end{equation}

We remark that as  a consequence of Theorem \ref{CLT}, we  deduce  the commutant lifting theorem for the elements of the  noncommutative varieties $\cV_J$  (\cite{Po-domains}).
More precisely, we can obtain the following result.

Let $J\neq F_{n_1}^\infty(\cD_f)$ be a WOT-closed two-sided ideal of  the noncommutative Hardy algebra $F_{n_1}^\infty(\cD_f)$  and let ${\bf B}_1:=(B_{1,1},\ldots, B_{1,n_1})$  and ${\bf C}_1:=(C_{1,1},\ldots, C_{1,n_1})$ be the corresponding constrained shifts acting
on $\cN_J$.  For each $j=1,2$, let $\cK_j$ be a Hilbert space and $\cE_j\subseteq \cN_J\otimes \cK_j$ be a co-invariant subspace  under each operator $B_{1,i}\otimes I_{\cK_j}$, \ $i=1,\ldots, n_1$.
If $X:\cE_1\to \cE_2$ is a bounded operator such that
\begin{equation*}
X[P_{\cE_1}(B_{1,i}\otimes I_{\cK_1})|_{\cE_1}]=[P_{\cE_2}(B_{1,i}\otimes I_{\cK_2})]|_{\cE_2}X,\quad i=1,\ldots,n_1,
\end{equation*}
then there exists
$G({\bf C}_1)\in \cR_{n_1}^\infty(\cV_J)\bar\otimes B(\cK_1,\cK_2)$
such that
$$
G({\bf C}_1)^*|_{\cE_2}=X^* \quad \text{ and }\quad \|G({\bf C}_1)\|=\|X\|.
$$

The analogue of Theorem \ref{dil} on noncommutative varieties $\cV_J(\cH)$  in the domain $\cD_f(\cH)$ is the following. Recall that $\cU_{\bf T}^\cK$ is the set of all unitary extensions of the isometry defined by  relation \eqref{iso2}.

\begin{theorem} \label{dil-com}  Let
  ${\bf T}_1:=(T_{1,1},\ldots, T_{1,n_1})$ and  ${\bf T}_1':=(T_{1,1}',\ldots, T_{1,n_1}')$ be elements of the noncommutative varieties $\cV_J(\cH)$ and $\cV_J(\cH')$, respectively,   and let ${\bf T}_2:=(T_{2,1},\ldots, T_{2,n_2})\in \cD_g(\cH', \cH)$ be  such that ${\bf T}_2\in \cI({\bf T}_1,{\bf T}_1')$.
Let ${\bf B}_1:=(B_{1,1},\ldots, B_{1,n_1})$ and ${\bf C}_1:=(C_{1,1},\ldots, C_{1, n_1})$ be the weighted creation operators associated with the noncommutative variety $\cV_J$.
If
$$\varphi_{U^*}({\bf \Lambda}_1)=[\varphi_{(\beta)}({\bf \Lambda}_1): \ \beta\in \FF_{n_2}^+, 1\leq |\beta|\leq k_2 ]
 $$
 is the contractive multi-analytic operator associated with $U\in \cU_{\bf T}^\cK$ and
   ${\bf T}_1$ is  pure in $\cD_f(\cH)$,   then the following relations hold:
$$
K_{J,{\bf T}_1'} T_{2,\beta}^* =\frac{1}{\sqrt{c_{\beta}}}\varphi_{(\beta)} ({\bf C}_1)^*K_{J,{\bf T}_1}, \qquad \beta\in \FF_{n_2}^+, 1\leq |\beta|\leq k_2,
$$
$$
K_{J,{\bf T}_1}T_{1,i}^*=
\left(B_{1,i}^*\otimes I_{\cD_{{\bf T}_1}}\right)  K_{J,{\bf T}_1},\quad
K_{J,{\bf T}_1'}(T_{1,i}')^*=
\left(B_{1,i}^*\otimes I_{\cD_{{\bf T}_1'}}\right)  K_{J,{\bf T}_1'}, \qquad i\in \{1,\ldots, n_1\},
$$
 where  $K_{J,{\bf T}_1}$ and $K_{J,{\bf T}_1'}$ are the  constrained  Poisson kernels associated with ${\bf T}_1$ and ${\bf T}_1'$, respectively.

\end{theorem}
\begin{proof} Since ${\bf T}_1\in \cV_J(\cH)$ and ${\bf T}_1'\in \cV_J(\cH')$,
the noncommutative Poisson kernels $K_{f,{\bf T}_1}$ and $K_{f,{\bf T}_1'}$ have ranges in $\cN_J\otimes  \cD_{{\bf T}_1}$ and $\cN_J\otimes  \cD_{{\bf T}_1'}$, respectively.
Due to Theorem \ref{dil}, we have
\begin{equation}\label{INT}
K_{f,{\bf T}_1'} T_{2,\beta}^* =\frac{1}{\sqrt{c_{\beta}}}\varphi_{(\beta)} ({\bf \Lambda}_1)^*K_{f,{\bf T}_1}, \qquad \beta\in \FF_{n_2}^+, 1\leq |\beta|\leq k_2.
\end{equation}
Since $\cN_J$ is co-invariant under $\Lambda_{1,1},\ldots, \Lambda_{1, n_1}$, we have
$$
\varphi_{(\beta)}({\bf \Lambda}_1)^*(\cN_J\otimes  \cD_{{\bf T}_1})\subset \cN_J\otimes  \cD_{{\bf T}_1'},  \qquad \beta\in \FF_{n_2}^+, 1\leq |\beta|\leq k_2,
$$
and
$$
 \varphi_{(\beta)}({\bf C}_1)=P_{\cN_J\otimes  \cD_{{\bf T}_1}} \varphi_j({\bf \Lambda}_1)|_{\cN_J\otimes  \cD_{{\bf T}_1'}}, \qquad  \qquad \beta\in \FF_{n_2}^+, 1\leq |\beta|\leq k_2,
 $$
 is a multi-analytic operator with respect to the universal model ${\bf B}$. Note that relation \eqref{INT} implies
 $$
 P_{\cN_J\otimes  \cD_{{\bf T}_1'}}K_{f,{\bf T}_1'} T_{2,\beta}^* =\frac{1}{\sqrt{c_{\beta}}}P_{\cN_J\otimes  \cD_{{\bf T}_1'}}\varphi_{(\beta)}({\bf \Lambda}_1)^*P_{\cN_J\otimes  \cD_{{\bf T}_1}}K_{f,{\bf T}_1},   \qquad \beta\in \FF_{n_2}^+, 1\leq |\beta|\leq k_2,
 $$
 which  proves that
 $$
K_{J,{\bf T}_1'} T_{2,\beta}^* =\frac{1}{\sqrt{c_{\beta}}}\varphi_{(\beta)} ({\bf C}_1)^*K_{J,{\bf T}_1}, \qquad \beta\in \FF_{n_2}^+, 1\leq |\beta|\leq k_2.
$$
 Since the other relations  in the theorem are due to \eqref{KJK}. The proof is complete.
 \end{proof}

 As a consequence of Theorem \ref{dil-com} we obtain the following  Schur \cite{Sc} type representation   for the  unit ball of $\cR_{n_1}^\infty(\cV_J)\bar \otimes B(\cH', \cH)$.

\begin{theorem} \label{transfer}  An operator $\Gamma:\cN_J\otimes \cH'\to \cN_J\otimes \cH$ is in the closed unit ball of $\cR_{n_1}^\infty(\cV_J)\bar \otimes B(\cH', \cH)$   if and only if there is a Hilbert space $\cE$ and a  unitary operator
$$
 \Omega=\left[ \begin{matrix} E&F\\G&H \end{matrix}\right]: \begin{matrix} \cH'\\ \oplus \\ \cE \end{matrix}\to   \begin{matrix} \cH\\ \oplus \\ \bigoplus\limits_{{\alpha\in \FF_{n_1}^+} \atop { 1\leq |\alpha|\leq k_1 }}\cE \end{matrix}
$$
such that $\Gamma=\text{\rm SOT-}\lim_{r\to 1}\varphi_{\Omega}(r{\bf C}_1)$, where

\begin{equation*}
\begin{split}
\varphi_{\Omega}(r{\bf C}_1)&:= I_{\cN_J}\otimes E+\left(I_{\cN_J}\otimes F\right) \left(I_{\cN_J\otimes \cH}-\sum\limits_{{\alpha\in \FF_{n_1}^+}\atop{ 1\leq |\alpha|\leq k_1}}
r^{|\alpha|}\sqrt{a_{ \alpha}} C_{1,\tilde\alpha}\otimes H_{(\alpha)}\right)^{-1}\\
& \qquad \qquad \times \left[\sqrt{a_{  \alpha}} C_{1,\tilde\alpha}\otimes I_\cH:\ \alpha\in \FF_{n_1}^+, 1\leq |\alpha|\leq k_1\right]\left(I_{\cN_J}\otimes G\right),
\end{split}
\end{equation*}
where ${\bf C}_1:=(C_{1,1},\ldots, C_{1,n_1})$ is the tuple of weighted right creation operators on $F^2(H_{n_1})$
 and $H$ has the operator  row matrix representation
 $$
 H=\left[ \begin{matrix} H_{(\alpha)}\\ \vdots\\ \alpha\in \FF_{n_1}^+, 1\leq |\alpha|\leq k_1 \end{matrix} \right]:\cE\to \bigoplus_{{\alpha\in \FF_{n_1}^+},{ 1\leq |\alpha|\leq k_1 }} \cE.
 $$

\end{theorem}

\begin{proof} Assume that $\Gamma:\cN_J\otimes \cH'\to \cN_J\otimes \cH$ is
a contractive multi-analytic operator  with respect to the universal model ${\bf B}_1:=(B_{1,1},\ldots, B_{1,n_1})$, i.e. $\Gamma(B_{1,i}\otimes I_{\cH'})=(B_{1,i}\otimes I_{\cH})\Gamma$ for any $i\in \{1,\ldots, n_1\}$. Due to the commutant lifting theorem for pure elements in $\cD_f$,   there exists
 a contractive multi-analytic  operator $\Psi:F^2(H_n)\otimes \cH'\to F^2(H_n)\otimes \cH$  with respect to the universal model ${\bf W}_1:=(W_{1,1},\ldots, W_{1,n_1})$, i.e.
 $\Psi(W_{1,i}\otimes I_{\cH'})=(W_{1,i}\otimes I_{\cH})\Psi$ for any $i\in \{1,\ldots, n_1\}$,  such that $\|\Gamma\|=\|\Psi\|$ and $\Psi^*|_{\cN_J\otimes \cH}=\Gamma^*$.

 Set ${\bf T}_1:=(W_{1,1}\otimes I_\cH,\ldots, W_{1,n_1}\otimes I_{\cH})$,
${\bf T}_1':=(W_{1,1}\otimes I_{\cH'},\ldots, W_{1,n_1}\otimes I_{\cH'})$,  $n_2=1$, and
${\bf T}_2:=\Psi$. Since $\Psi\in \cI({\bf T}_1, {\bf T}'_1)$,  Theorem \ref{dil}  and  Lemma \ref{strong-limit} show that there is a unitary operator
$$
 \Omega=\left[ \begin{matrix} E&F\\G&H \end{matrix}\right]: \begin{matrix} \cD_{{\bf T}_1'}\\ \oplus \\ \cE \end{matrix}\to   \begin{matrix} \cD_{{\bf T}_1}\\ \oplus \\ \bigoplus\limits_{{\alpha\in \FF_{n_1}^+}, { 1\leq |\alpha|\leq k_1 }}\cE \end{matrix}
$$
 such that $\varphi_{\Omega}({\bf \Lambda}_1):=\text{\rm SOT-}\lim_{r\to 1}\varphi_{\Omega}(r{\bf \Lambda}_1)$ is a multi-analytic operator in $\cR_{{n}_1}^\infty(\cD_f)\bar\otimes B(\cD_{{\bf T}_1'},\cD_{{\bf T}_1})$, where $\varphi_{\Omega}(r{\bf \Lambda}_1)$ is  defined as in the theorem and such that
$K_{f,{\bf T}_1} \Psi^*=\varphi_{\Omega}({\bf \Lambda}_1)^*K_{f,{\bf T}_1}$.
Due to relation
$I-\sum_{ |\beta|\geq 1} a_\beta W_\beta W_\beta^*=P_\CC$, we deduce that
 $\cD_{{\bf T}_1}=\cH$ and  $\cD_{{\bf T}_1'}=\cH'$.
On the other hand, since
\begin{equation*}
P_\CC W_\beta^* e_\alpha =\begin{cases}
\frac {1}{\sqrt{b_{\beta}}}   & \text{ if } \alpha=\beta\\
0& \text{ otherwise},
\end{cases}
\end{equation*}
one can easily see that
  the noncommutative Poisson kernel $K_{f,{\bf T}_1}$ is the identity on
$F^2(H_{n_1})\otimes \cH$. Consequently, $\Psi=\varphi_{\Omega}({\bf \Lambda}_1)$.
Since $\Psi^*|_{\cN_J\otimes \cH}=\Gamma^*$, we deduce that $\Gamma=\varphi_\Omega({\bf C}_1)$.

 To prove the converse, note that,  in the  particular case when $n_2=1$ and $U^*=\Omega$,   Lemma \ref{strong-limit} shows that
 $\Psi:=\text{\rm SOT-}\lim_{r\to 1}\varphi_{\Omega}(r{\bf \Lambda}_1)$, where
\begin{equation*}
\begin{split}
\varphi_{\Omega}(r{\bf \Lambda}_1):= I_{F^2(H_{n_1})}\otimes E&+\left(I_{F^2(H_{n_1})}\otimes F\right) \left(I_{F^2(H_{n_1})\otimes \cE}-\sum\limits_{{\alpha\in \FF_{n_1}^+}\atop{ 1\leq |\alpha|\leq k_1}}
r^{|\alpha|}\sqrt{a_{ \alpha}} \Lambda_{1,\tilde\alpha}\otimes H_{(\alpha)}\right)^{-1}\\
& \times \left[\sqrt{a_{  \alpha}} \Lambda_{1,\tilde\alpha}\otimes I_\cH:\ \alpha\in \FF_{n_1}^+, 1\leq |\alpha|\leq k_1\right]\left(I_{F^2(H_{n_1})}\otimes G\right)
\end{split}
\end{equation*}
 and $H=\left[ \begin{matrix} H_{(\alpha)}\\ \vdots\\ \alpha\in \FF_{n_1}^+, 1\leq |\alpha|\leq k_1 \end{matrix} \right]:\cE\to \bigoplus_{{\alpha\in \FF_{n_1}^+} \atop { 1\leq |\alpha|\leq k_1 }} \cE$, is a contractive multi-analytic operator with respect to ${\bf W}_1$.
 Since $\cN_J$ is a  co-invariant subspace under $\Lambda_{1,1},\ldots, \Lambda_{1,n_1}$, we deduce that  $\Gamma:=\varphi_\Omega({\bf C}_1)=P_{\cN_J\otimes \cD_{{\bf T}_1'}}\Psi|_{\cN_J\otimes \cD_{{\bf T}_1}}$  is a contractive multi-analytic operator with respect to ${\bf B}_1$.
 The proof is complete.
\end{proof}

\smallskip

\section{And\^ o type dilations and inequalities on  noncommutative bi-domains and varieties}

In this section, we obtain And\^ o type dilations and inequalities  for the elements of the bi-domain ${\bf D}_{(f,g)}$ and a class of noncommutative varieties. The commutative case as well the matrix case  are  also discussed.

We recall that, given a positive regular formal power series
$g=\sum\limits_{{\beta\in \FF_{n_2}^+},{ 1\leq |\beta|\leq k_2}} c_\beta X_\beta$, the noncommutative ellipsoid $\cE_g(\cH)\supseteq \cD_g(\cH)$ is defined  by
$
\cE_g(\cH):=\left\{ {\bf X}:=(X_1,\ldots, X_{n_2}): \  \sum_{|\beta|=1} c_\beta X_\beta X_\beta^*\leq I\right\}.
$

One of the most important consequences of the results from Section 2 is the following And\^ o type dilation for  the bi-domain ${\bf D}_{(f,g)}(\cH):=\cD_f(\cH)\times_c \cD_g(\cH)$, where $f$ and $g$ are positive regular noncommutative polynomials, and for  the noncommutative variety
$${\bf D}_{(f,g)}^J(\cH):=\left\{({\bf T}_1, {\bf T}_2)\in {\bf D}_{(f,g)}(\cH): {\bf T}_1\in \cV_J(\cH)\right\}.
$$
 We recall that $\cU_{\bf T}^\cK$ is the set of all unitary extensions of the isometry defined by relation \eqref{iso2}.  According to Lemma \ref{strong-limit}, for each $U\in \cU_{\bf T}^\cK$,
 the strong operator topology limit
$\varphi_{U^*}({\bf \Lambda}_1):=\text{\rm SOT-}\lim_{r\to 1}\varphi_{U^*}(r{\bf \Lambda_1})
$
 exists and defines a contractive  multi-analytic  operator.

\begin{theorem} \label{dil2}
Let ${\bf T}=({\bf T}_1, {\bf T}_2)\in {\bf D}^J_{(f,g)}(\cH)$ with ${\bf T}_1=(T_{1,1},\ldots, T_{1,n_1})$ and ${\bf T}_2:=(T_{2,1},\ldots, T_{2,n_2})$.     If
 $$\varphi_{U^*}({\bf \Lambda}_1)=(\varphi_{(\beta)}({\bf \Lambda}_1): \ \beta\in \FF_{n_2}^+, 1\leq |\beta|\leq k_2 )
 $$ is the contractive  multi-analytic operator associated with $U\in \cU_{\bf T}^\cK$,
then
$$
K_{J,{\bf T}_1} T_{1,\alpha}^* T_{2,\beta}^*=\left(B_{1,\alpha}^*\otimes I_{\cD_{{\bf T}_1}}\right)\psi_\beta({\bf C}_1)^* K_{J,{\bf T}_1}, \qquad \alpha\in \FF_{n_1}^+, \beta\in \FF_{n_2}^+,
$$
  where
  \begin{enumerate}
  \item[(i)] ${\bf B}_1:=(B_{1,1},\ldots, B_{1,n_1})$ and ${\bf C}_1:=(C_{1,1},\ldots, C_{1,n_1})$ are the constrained creation operators associated with   the variety $\cV_J$;
      \item[(ii)]  $K_{J,{\bf T}_1}$ is  the constrained Poisson kernel  associated $\cV_J$;
          \item[(iii)] $\psi({\bf C}_1):=(\psi_1({\bf C}_1),\ldots, \psi_{n_2}({\bf C}_1))\in \cE_g(\cN_J\otimes \cD_{{\bf T}_1})$, where
          $$\psi_j({\bf C}_1):=\frac{1}{\sqrt{c_{g_j}}}\varphi_{(g_j)}({\bf C}_1), \qquad j\in \{1,\ldots, n_2\}.
          $$
          \end{enumerate}
\end{theorem}
\begin{proof} In the particular case when ${\bf T}_1={\bf T}_1'$,  Theorem \ref{dil-com} shows that
$
K_{J,{\bf T}_1} T_{2,j}^* =\psi_j({\bf C}_1)^*K_{J,{\bf T}_1}$ for  $j\in\{1,\ldots, n_2\}
$
and
$
K_{J,{\bf T}_1}T_{1,i}^*=
\left(B_{1,i}^*\otimes I_{\cD_{{\bf T}_1}}\right)  K_{J,{\bf T}_1}$ for  $ i\in \{1,\ldots, n_1\}.
$
Hence, the relation in the theorem follows.
\end{proof}

We remark that Theorem \ref{dil2} provides a model  and a characterization of the elements $({\bf T}_1, {\bf T}_2)\in   \cD_f(\cH)\times_c \cE_g(\cH)$ with ${\bf T}_1\in \cV_J(\cH)$ .
Indeed, if  ${\bf T}=({\bf T}_1, {\bf T}_2)\in B(\cH)^{n_1}\times_c B(\cH)^{n_2}$, then ${\bf T}\in   \cD_f(\cH)\times_c \cE_g(\cH)$ with ${\bf T}_1\in \cV_J(\cH)$ if and only if there is a Hilbert space $\cD$, a multi-analytic operator (with respect to ${\bf B}_1$)
 $$
 \psi({\bf C}_1)=(\psi_1({\bf C}_1),\ldots, \psi_{n_2}({\bf C}_1))\in \cE_g(\cN_J\otimes \cD),
$$
and a co-invariant subspace $\cM\subset \cN_J\otimes \cD$ under each of the operators
$B_{1,i}\otimes I_\cD$ and  $\varphi_j({\bf C}_1)$, where $i\in \{1,\ldots, n_1\}$ and
$j\in \{1,\ldots, n_2\}$, such that $\cM$ can be identified with $\cH$,
$$
(B_{1,i}^*\otimes I_\cD)|_\cH=T_{1,i}^*,\quad \text{and} \quad \varphi_j({\bf C}_1)^*|_\cH=T_{2,j}^*.
$$
Note that the direct implication  is due to Theorem \ref{dil2} under the identification of $\cH$ with $K_{J,{\bf T}_1}\cH$. The converse is obvious.

In what follows, we obtain And\^ o type inequalities  for the bi-domain ${\bf D}_{(f,g)}(\cH)$ and the noncommutative variety
${\bf D}_{(f,g)}^J(\cH)$.
 First, we consider the case when ${\bf T}_1=(T_{1,1},\ldots, T_{1,n_1})$ and ${\bf T}_2:=(T_{2,1},\ldots, T_{2,n_2})$  have  the property that ${\bf T}=({\bf T}_1, {\bf T}_2)\in {\bf D}_{(f,g)}^J(\cH)$ with
$d_i:=\dim \cD_{{\bf T}_i}<\infty$   and
$
  d_1+m_1d_2=m_2 d_1+d_2,
  $
  where
  $$
  m_i:=\text{\rm card} \{\alpha\in \FF_{n_j}^+: \ 1\leq|\alpha|\leq k_j\},\qquad j=1,2.
  $$
  The set $\cU_{\bf T}$  consists  of  unitary extensions
   $U:\cD_{{\bf T}_1}\oplus \bigoplus_{\alpha\in \FF_{n_1}^+, 1\leq |\alpha|\leq k_1}\cD_{{\bf T}_2}\to  \bigoplus_{\beta\in \FF_{n_2}^+, 1\leq |\beta|\leq k_2}\cD_{{\bf T}_1}\oplus \cD_{{\bf T}_2}
$ of the isometry
\begin{equation}\label{UU}
U\left(\Delta_{{\bf T}_1}h\oplus \bigoplus_{{\alpha\in \FF_{n_1}^+, 1\leq |\alpha|\leq k_1}} \sqrt{a_\alpha}\Delta_{{\bf T}_2}T_{1,\alpha}^*h \right):=
\bigoplus_{\beta\in \FF_{n_2}^+, 1\leq |\beta|\leq k_2}\sqrt{c_\beta}\Delta_{{\bf T}_1}T_{2,\beta}^*h ,\qquad h\in \cH.
\end{equation}
Let ${\bf Z}:=\left< Z_1,\ldots, Z_{n_1}\right>$ and
${\bf Z}':=\left< Z_1',\ldots, Z_{n_2}'\right>$  be noncommutative indeterminates and assume that
 $Z_iZ_j'=Z_j'Z_i$ for any $i\in \{1,\ldots, n_1\}$ and $j\in \{1,\ldots, n_2\}$. We denote by  $\CC\left<{\bf Z}, {\bf Z}'\right>$  the complex algebra of all  polynomials in  indeterminates    $Z_{1},\ldots, Z_{n_1}$ and $Z_{1}',\ldots, Z_{n_2}'$. Note that when $n_1=n_2=1$, then $\CC\left<{\bf Z}, {\bf Z}'\right>$ coincides with the algebra $\CC[z,w]$ of complex polynomials in two variable.

\begin{theorem}\label{ando}
Let ${\bf T}=({\bf T}_1, {\bf T}_2)\in {\bf D}_{(f,g)}^J(\cH)$  with  ${\bf T}_1=(T_{1,1},\ldots, T_{1,n_1})$ and ${\bf T}_2:=(T_{2,1},\ldots, T_{2,n_2})$   such that
$$d_i:=\dim \cD_{{\bf T}_i}<\infty\ \text{ and } \ d_1+m_1d_2=d_2+m_2d_1,
$$
and  let ${\bf B}_1:=(B_{1,1},\ldots, B_{1,n_1})$ and ${\bf C}_1:=(C_{1,1},\ldots, C_{1,n_1})$ be the constrained weighted  creation operators associated with   the   noncommutative variety $\cV_J$.
If  $U\in \cU_{\bf T}$,   then
$$
\|[p_{rs}({\bf T}_1,{\bf T}_2)]_{k}\|\leq   \|[p_{rs}({\bf B}_1\otimes I_{\CC^{d_1}}, \psi ({\bf  C}_1))]_{k}\|, \qquad [p_{rs}]_{ k}\in M_k(\CC\left<{\bf Z}, {\bf Z}'\right>), k\in \NN,
$$
where  $\psi({\bf C}_1)=(\psi_1({\bf C}_1),\ldots, \psi_{n_2}({\bf C}_1))$ is  uniquely determined by $U$ as in Theorem \ref{dil2} and each $\psi_j({\bf C}_1) $ is a $d_1\times d_1$-matrix with entries in the Hardy algebra
 $\cR_{n_1}^\infty(\cV_J)$.
 \end{theorem}

 \begin{proof} Since $d_1+m_1d_2=d_2+m_2d_1$, the set $\cU_{\bf T}$ of all unitary extensions of the isometry $U$ defined by relation \eqref{UU} is non-empty. Fix any $U\in\cU_{\bf T}$ and apply Theorem \ref{dil2} to ${\bf T}=({\bf T}_1, {\bf T}_2)\in {\bf D}_{(f,g)}^J(\cH)$  and $U\in \cU_{\bf T}$.
  Then  we deduce that
 $$
K_{J,{\bf T}_1} T_{1,\alpha}^* T_{2,\beta}^*=\left(B_{1,\alpha}^*\otimes I_{\cD_{{\bf T}_1}}\right)\psi_\beta({\bf C}_1)^* K_{J,{\bf T}_1}
$$
for any $\alpha\in \FF_{n_1}^+$ and $\beta\in \FF_{n_2}^+$. Consequently, if $p$ is any polynomial in $\CC\left<{\bf Z}, {\bf Z}'\right>$, we obtain
$$
K_{J,{\bf T}_1} p({\bf T}_1,{\bf T}_2)=p({\bf B}_1\otimes I_{\CC^{d_1}}, \psi ({\bf  C}_1))K_{J,{\bf T}_1}.
$$
Since ${\bf T}_1\in \cV_J(\cH)$, the noncommutative Poisson kernel $K_{J,{\bf T}_1}$ is an isometry, which implies
$$
 p({\bf T}_1,{\bf T}_2)=K_{J,{\bf T}_1}^*p({\bf B}_1\otimes I_{\CC^{d_1}}, \psi ({\bf  C}_1))K_{J,{\bf T}_1}.
$$
Now, it is clear that
$$
\|[p_{rs}({\bf T}_1,{\bf T}_2)]_{k\times k}\|\leq   \|[p_{rs}({\bf B}_1\otimes I_{\CC^{d_1}}, \psi ({\bf  C}_1))]_{k\times k}\|, \qquad [p_{rs}]_{k\times k}\in M_k(\CC\left<{\bf Z}, {\bf Z}'\right>), k\in \NN.
$$
 The proof is complete.
 \end{proof}

Denote by $\cQ_{\bf n}^*$ the set of all formal polynomials of the form
$q({\bf Z}, {\bf Z}')=\sum a_{\alpha,\beta, \gamma, \sigma} Z_\alpha Z_\beta' (Z_\sigma')^*Z_\gamma^* $, with complex coefficients, where ${\bf Z}:=\left< Z_1,\ldots, Z_{n_1}\right>$ and
${\bf Z}':=\left< Z_1',\ldots, Z_{n_2}'\right>$. In what follows, we show that
if we drop the conditions $d_i:=\dim \cD_{{\bf T}_i}<\infty$ and $d_1+m_1d_2=d_2+m_2d_1$, in Theorem \ref{ando},  we can obtain the following And\^ o type inequality.

\begin{theorem} \label{ando1} Let ${\bf T}=({\bf T}_1, {\bf T}_2)\in {\bf D}_{(f,g)}^J(\cH)$  with  ${\bf T}_1=(T_{1,1},\ldots, T_{1,n_1})$ and ${\bf T}_2:=(T_{2,1},\ldots, T_{2,n_2})$.
      If $U\in \cU_{{\bf T}}^\cK$, then
$$
\|[q_{rs}({\bf T}_1,{\bf T}_2)]_{k}\|\leq   \|[q_{rs}({\bf V}_1,{\bf V}_2)]_{k}\|, \qquad [p_{rs}]_{k\times k}\in M_k(\cQ_{\bf n}^*),
$$
where ${\bf V}_1:={\bf B}_1\otimes I_{\cD_{{\bf T}_1}}$ and  ${\bf V}_2:= \psi({\bf  C}_1)$ is  uniquely determined by $U$ as in Theorem \ref{dil2}.
\end{theorem}
\begin{proof} The proof uses Theorem \ref{dil2} and is similar to the proof of Theorem \ref{ando}. We shall omit it.
\end{proof}

For any  polynomial $p\in \CC\left<{\bf Z},{\bf Z}\right>$,  define
$
\|p\|_u:=\sup \|p({\bf T}_1, {\bf T}_2)\|,
$
where the supremum is taken over all pairs
$({\bf T}_1, {\bf T}_2)\in \cD_f(\cH)\times_c \cE_g(\cH)$ and any Hilbert space $\cH$.
Then  $\|\cdot\|_u$ defines an algebra  norm on $\CC\left<{\bf Z},{\bf Z}'\right>$. Since the proof is very similar to that of Lemma 2.4 from \cite{Po-Ando}, we  omit it.
 If for $[p_{ij}]_k\in M_k(\CC\left<{\bf Z},{\bf Z}\right>)$, we set
$$
\|[p_{ij}]\|_{u,k}:=\|[p_{ij}]_k\|_u:=\sup \|[p_{ij}({\bf T}_1, {\bf T}_2)]\|,
$$
where the supremum is taken over all pairs
$({\bf T}_1, {\bf T}_2)\in \cD_f(\cH)\times_c \cE_g(\cH)$ and any Hilbert space $\cH$, we obtain a sequence of norms on the matrices over $\CC\left<{\bf Z},{\bf Z}'\right>$. We call $\left(\CC\left<{\bf Z},{\bf Z}'\right>, \|\cdot\|_{u,k}\right)$ the universal operator algebra for the
 bi-domain $ \cD_f(\cH)\times_c \cE_g(\cH)$.

In what follows, we  prove that the abstract bi-domain
$$
\cD_f\times_c \cE_g:=\{\cD_f(\cH)\times_c \cE_g(\cH): \ \cH \text{ is a Hilbert space}\}
$$
has a universal model $({\bf W}_1\otimes I_{\ell^2},  \psi({\bf  \Lambda}_1))$,
 where ${\bf W}_1=(W_{1,1},\ldots, W_{1,n_1})$ and ${\bf \Lambda}_1=(\Lambda_{1,1},\ldots, \Lambda_{1,n_1})$ are the weighted left and right creation operators  associated with  the regular domain $\cD_f$, respectively, and
$$
\psi({\bf  \Lambda}_1)=(\psi_1({\bf \Lambda}_1),\ldots, \psi_{n_2}({\bf \Lambda}_1))\in \cE_g(F^2(H_{n_1})\otimes \ell^2)
$$
 is a certain multi-analytic operator with respect to ${\bf W}_1$.

\begin{theorem} \label{ando11}  There is a multi-analytic operator
$ \psi({\bf  \Lambda}_1)=(\psi_1({\bf \Lambda}_1),\ldots, \psi_{n_2}({\bf \Lambda}_1))\in \cE_g(F^2(H_{n_1})\otimes \ell^2)$
such that
$$
\|[p_{rs}({\bf T}_1,{\bf T}_2)]_{k}\|\leq  \|[p_{rs}({\bf W}_1\otimes I_{\ell^2},  \psi({\bf  \Lambda}_1))]_{k}\|, \qquad  p_{rs}\in \CC\left<{\bf Z}, {\bf Z}'\right>,
$$
 for any  $({\bf T}_1, {\bf T}_2)\in \cD_f(\cH)\times_c \cE_g(\cH)$ and any $k\in \NN$.
\end{theorem}

\begin{proof}
Given a  matrix $[p_{ij}]_k\in M_k(\CC\left<{\bf Z},{\bf Z}'\right>)$, we have
$
\|[p_{ij}]_k\|_u:=\sup \|[p_{ij}({\bf T}_1, {\bf T}_2)]_k\|,
$
where the supremum is taken over all pairs
$({\bf T}_1, {\bf T}_2)\in  \cD_f(\cH)\times_c \cE_g(\cH)$ and any Hilbert space $\cH$. Using a standard argument, one can prove that the supremum is the same if we consider  only infinite dimensional separable Hilbert spaces.  Since $r{\bf T}_1$ is a pure element in $\cD_f(\cH)$ for any $r\in [0,1)$, it is clear that
$$
\|[p_{ij}]_k\|_u=\sup_{{({\bf T}_1, {\bf T}_2)\in  \cD_f\times_c \cE_g}\atop{ {\bf T}_1 \text{ pure}}} \|[p_{ij}({\bf T}_1, {\bf T}_2)]_k\|.
$$
Fix $[p_{ij}]_k\in M_k(\CC\left<{\bf X}, {\bf Y}\right>)$ and  choose a sequence $\left\{({\bf T}_1^{(m)}, {\bf T}_2^{(m)})\right\}_{m=1}^\infty$  in $ \cD_f(\cH)\times_c \cE_g(\cH)$ with $\cH$ separable and ${\bf T}_1^{(m)}$ pure element  in $\cD_f(\cH)$, and  such that
\begin{equation}
\label{sup2}
\|[p_{ij}]_k\|_u=\sup_{m}\|[p_{ij}({\bf T}_1^{(m)}, {\bf T}_2^{(m)})]_k\|.
\end{equation}
 Using Theorem \ref{ando1}, in the particular case when $J=\{0\}$,  we find,  for each $m\in \NN$,  a multi-analytic operator $\psi^{(m)}({\bf  \Lambda}_1):=(\psi_1^{(m)}({\bf  \Lambda}_1),\ldots, \psi_{n_2}^{(m)}({\bf  \Lambda}_1))$ with respect to ${\bf W}_1$,  which belongs to the  ellipsoid  $\cE_g(F^2(H_{n_1})\otimes \CC^{d(m)})$, where $d(m):=\cD_{{\bf T}_1^{(m)}}$, such that
$$
\|[p_{ij}({\bf T}_1^{(m)},{\bf T}_2^{(m)})]_k\|\leq  \|[p_{ij}({\bf W}_1\otimes I_{\CC^{d(m)}},\psi^{(m)}({\bf \Lambda}_1))]_{k }\|.
$$
Consequently, setting $\oplus_{m=1}^\infty {\bf T}_1^{(m)}:=\left(\oplus_{m=1}^\infty {T}_{1,1}^{(m)},\ldots,\oplus_{m=1}^\infty { T}_{1,n_1}^{(m)}\right)\in \cD_f(\oplus_{m=1}^\infty \cH)$,  relation \eqref{sup2} implies
\begin{equation*}
\begin{split}
\|[p_{ij}]_k\|_u&=\|[p_{ij}(\oplus_{m=1}^\infty {\bf T}_1^{(m)}, \oplus_{m=1}^\infty {\bf T}_2^{(m)})]_k\|\\
&\leq
\|[p_{ij}(\oplus_{m=1}^\infty  ({\bf W}_1\otimes I_{\CC^{d(m)}}), \oplus_{m=1}^\infty \psi^{(m)}({\bf \Lambda}_1))]_k\|\leq \|[p_{ij}]_k\|_u.
\end{split}
\end{equation*}
 This shows  that
\begin{equation}
\label{ma22}
\|[p_{ij}]_k\|_u= \|[p_{ij}({\bf W}_1\otimes I_{ \ell^2},\zeta({\bf \Lambda}_1))]_{k }\|,
\end{equation}
where $\zeta({\bf \Lambda}_1)=(\zeta_1({\bf \Lambda}_1),\ldots, \zeta_{n_2}({\bf \Lambda}_1)):=\oplus_{m=1}^\infty \psi^{(m)}({\bf \Lambda}_1))\in \cE_g(F^2(H_{n_1})\otimes \ell^2)$ is
a  multi-analytic operator with respect to ${\bf W}_1$.
Let $\CC_\QQ\left<{\bf Z}, {\bf Z}'\right>$ be the set of all polynomials with coefficients in $\QQ+i\QQ$, and let
$
[p^{(1)}_{ij}]_k, [p^{(2)}_{ij}]_k, \ldots$ be an enumeration of  the set $\{[p_{ij}]_k: \ p_{ij}\in \CC_\QQ\left<{\bf Z}, {\bf Z}'\right>\}$.
Due to relation \eqref{ma22}, for each $s\in \NN$, there is a multi-analytic operator  $\zeta^{(s)}({\bf \Lambda}_1)=(\zeta^{(s)}_1({\bf \Lambda}_1),\ldots, \zeta^{(s)}_{n_2}({\bf \Lambda}_1))\in \cE_g(F^2(H_{n_1})\otimes \ell^2)$ such that
\begin{equation}
\label{ma222}
\|[p^{(s)}_{ij}]_k\|_u= \|[p_{ij}^{(s)}({\bf W}_1\otimes I_{ \ell^2},\zeta^{(s)}({\bf \Lambda}_1))]_{k }\|,\qquad s\in \NN.
\end{equation}
Define the multi-analytic operator $\Omega_k({\bf \Lambda}_1):=\oplus_{s=1}^\infty \zeta^{(s)}({\bf \Lambda}_1)\in \cE_g(F^2(H_{n_1}\otimes \ell^2)$  and let us prove that
\begin{equation}
\label{qqij}
\|[q_{ij}]_k\|_u= \|[q_{ij}({\bf W}_1\otimes I_{ \ell^2},\Omega_k({\bf \Lambda}_1))]_{k }\|
\end{equation}
for any $[q_{ij}]_k\in M_k(\CC\left<{\bf Z}, {\bf Z}'\right>)$.
Note that relation \eqref{ma222} implies
\begin{equation}
\label{Ga2}
\|[p^{(s)}_{ij}]_k\|_u= \|[p_{ij}^{(s)}({\bf W}_1\otimes I_{ \ell^2},\Omega_k({\bf \Lambda}_1))]_{k }\|\qquad \text{for any } s\in \NN.
\end{equation}
 Fix $[q_{ij}]_k\in M_k(\CC\left<{\bf Z}, {\bf Z}'\right>)$ and $\epsilon >0$,  and choose $[p^{(s_0)}_{ij}]_k$ such that
 \begin{equation}\label{u2}
 \left\|[q_{ij}]_k-[p^{(s_0)}_{ij}]_k\right\|_u<\epsilon.
 \end{equation}
 Using relations \eqref{ma22}, \eqref{u2}, and \eqref{Ga2}, we deduce that there is $\zeta^{(q)}:=(\zeta_1^{(q)}({\bf \Lambda}_1),\ldots, \zeta_{n_2}^{(q)}({\bf \Lambda}_1))$ in the ellipsoid  $ \cE_g(F^2(H_{n_1})\otimes \ell^2)$ such that
 \begin{equation*}
 \begin{split}
 \|[q_{ij}]_k\|_u&=\|[q_{ij}({\bf W}_1\otimes I_{ \ell^2},\zeta^{(q)}({\bf \Lambda}_1))]_{k }\|
\leq\|[p_{ij}^{(s_0)}({\bf W}_1\otimes I_{ \ell^2},\zeta^{(q)}({\bf \Lambda}_1))]_{k }\| +\epsilon\\
 &\leq \|[p_{ij}^{(s_0)}]_k\|_u+\epsilon
 =
 \|[p_{ij}^{(s_0)}({\bf W}_1\otimes I_{ \ell^2},\Omega_k({\bf \Lambda}_1))]_{k }\|+\epsilon\\
 &\leq
 \|[q_{ij}({\bf W}_1\otimes I_{ \ell^2},\Omega_k({\bf \Lambda}_1))]_{k }\|+2\epsilon
 \end{split}
 \end{equation*}
for any $\epsilon>0$, which proves  relation \eqref{qqij}.
Note  that  $\psi({\bf \Lambda}_1):=\oplus_{k=1}^\infty \Omega_k({\bf \Lambda}_1)$ is
 a   multi-analytic operator which belongs to  the ellipsoid $\cE_g(F^2(H_{n_1})\otimes \ell^2)$  and
$
\|[q_{ij}]_k\|_u= \|[q_{ij}({\bf W}_1\otimes I_{ \ell^2},\psi({\bf \Lambda}_1))]_{k }\|
$
for any $[q_{ij}]_k\in M_k(\CC\left<{\bf Z}, {\bf Z}'\right>)$ and any $k\in \NN$.
The proof is complete.
\end{proof}

Theorem \ref{ando11}  shows that $\left(\CC\left<{\bf Z},{\bf Z}'\right>, \|\cdot\|_{u,k}\right)$ can be realized completely isometrically isomorphic as a concrete algebra of operators. The closed non-self-adjoint algebra generated by the operators
$ W_{1,1}\otimes I_{\ell^2},\ldots, W_{1,n_1}\otimes I_{\ell^2}, \psi_1({\bf \Lambda}_1),\ldots, \psi_{n_2}({\bf \Lambda}_1)$ and the identity is denoted by
$\cA(\cD_f\times_c\cE_g)$  and  can be seen as the universal operator algebra of the  bi-domain $\cD_f\times_c\cE_g$.

We remark that the noncommutative variety $\cV_J\times_c \cE_g$
 also has a universal model. Similarly to the proof of Theorem \ref{ando11}, one can show that  there is   a multi-analytic operator  $\psi({\bf  C}_1)=(\psi_1({\bf  C}_1),\ldots, \psi_{n_2}({\bf  C}_1))$,  with respect to ${\bf B}_1$,  in  $ \cE_g(\cN_J\otimes \ell^2)$
such that
$$
\|[p_{rs}({\bf T}_1,{\bf T}_2)]_{k}\|\leq  \|[p_{rs}({\bf B}_1\otimes I_{\ell^2},  \psi({\bf  C}_1))]_{k}\|, \qquad  p_{rs}\in \CC\left<{\bf Z}, {\bf Z}'\right>,
$$
 for any  $({\bf T}_1, {\bf T}_2)\in \cV_J(\cH)\times_c \cE_g(\cH)$ and  any $k\in \NN$.

In the end of this section, we discuss the commutative case.
Let  $J_c$ be  the $WOT$-closed two-sided ideal of the Hardy algebra $F_{n_1}^\infty(\cD_f)$ generated by the commutators $W_j W_i-W_iW_j$ for  $i,j\in \{1,\ldots, n_1\}$.
 Note that the variety $\cV_{J_c}(\cH)$ consists of all pure tuples $(X_1,\ldots, X_{n_1})\in \cD_f(\cH)$ with  commuting entries.
  The Hardy algebra $F_{n_1}^\infty(\cV_{J_c})$ is the $WOT$-closed algebra generated by the compressions  $L_i:=P_{F_s^2(\cD_f)} W_i|_{F_s^2(\cD_f)}$, $i=1,\ldots, n_1$, and the identity, where $F_s^2(\cD_f)=\cN_{J_c}$ is the   symmetric weighted Fock space
associated with the noncommutative domain $\cD_f$.  In
\cite{Po-domains},
we    prove that  $F_s^2(\cD_f)$
can be identified with  a Hilbert space $H^2(\cD_f^\circ(\CC))$ of
holomorphic functions defined on the scalar domain
$$
\cD_f^\circ(\CC):=\left\{ (\lambda_1,\ldots, \lambda_{n_1})\in \CC^{n_1}: \
\sum_{|\alpha|\geq 1} a_\alpha |\lambda_\alpha|^2<1\right\},
$$
namely,  the reproducing kernel
Hilbert space with reproducing kernel $\kappa_f:\cD_f^\circ(\CC)\times
\cD_f^\circ(\CC)$ defined by
$
\kappa_f(\mu,\lambda):=\frac{1}{1-\sum_{|\alpha|\geq 1} a_\alpha
\mu_\alpha \overline{\lambda}_\alpha}$ for $\mu,\lambda\in
C$.
 We also  identified the algebra of all
multipliers of the Hilbert space $H^2(\cD_f^\circ(\CC))$ with the Hardy algebra
$F_{n_1}^\infty(\cV_{J_c})$. Under this identification, $L_i$ is the multiplier $M_{\lambda_i}$ by the coordinate function.
We denote ${\bf M}_{\lambda,n_1}:=(M_{\lambda_1},\ldots, M_{\lambda_{n_1}})$. Similarly, one can identify the Hardy algebra $R_{n_1}^\infty(\cV_{J_c})$ with
the algebra of all
multipliers of the Hilbert space $H^2(\cD_{\tilde f}^\circ(\CC))$, where $\tilde{f} :=\sum_{|\alpha|\geq 1} a_{\tilde \alpha} Z_\alpha$. Note also that $\cD_f^\circ(\CC)=\cD_{\tilde f}^\circ(\CC)$.
\begin{theorem}\label{commutative} Let $f\in \CC\left<{\bf Z}\right>$ and $g\in \CC\left<{\bf Z}'\right>$ be two positive regular noncommutative polynomials and let
$({\bf T}_1, {\bf T}_2)\in \cD_f(\cH)\times_c \cD_g(\cH)$ be such that
each tuple ${\bf T}_j=(T_{j,1},\ldots, T_{j,n_j})$ has commuting entries and $d_j:=\rank \Delta_{{\bf T}_j}$, $j=1,2$. Then there exist multipliers $M_{\Phi_f}$ and $M_{\Phi_g}$ of $H^2(\cD_f^\circ)\otimes \CC^{d_1}$ and $H^2(\cD_g^\circ)\otimes \CC^{d_2}$, respectively, such that $M_{\Phi_f}\in \cE_f(H^2(\cD_f^\circ))$, $M_{\Phi_g}\in \cE_g(H^2(\cD_g^\circ))$, and
$$
\|[p_{rs}({\bf T}_1,{\bf T}_2)]_{k}\|\leq \min \left\{ \|[p_{rs}({\bf M}_{\lambda, n_1}\otimes I_{\CC^{d_1}},  M_{\Phi_f})]_{k}\|,  \|[p_{rs}(M_{\Phi_g}, {\bf M}_{\lambda, n_2}\otimes I_{\CC^{d_2}})]_{k}\|\right\}
$$
for any matrix $ [p_{rs}]_k\in M_k(\CC\left<{\bf Z}, {\bf Z}'\right>)$
 and any    $k\in \NN$.
\end{theorem}
\begin{proof} Applying Theorem \ref{ando1} to the pairs
$({\bf T}_1, {\bf T}_2)\in {\bf D}_{(f,g)}^{J_c}(\cH)$ and $({\bf T}_2, {\bf T}_1)\in {\bf D}_{(g,f)}^{J_c}(\cH)$ and using the   the identifications preceding this theorem, one can easily complete the proof.
\end{proof}
We should mention that all the results of the present paper concerning And\^ o type dilations and inequalities can be written in the commutative multivariable  setting of Theorem \ref{commutative}. Moreover,  in  the particular case when $n_1=n_2=1$, we obtain  extensions of And\^ o's results \cite{An}, Agler-McCarthy's inequality \cite{AM}, and Das-Sarkar extension \cite{DS},  to  larger classes of commuting operators.

A few remarks concerning the matrix case when $n_1=n_2=1$ are necessary. If $T_1$ and $T_2$  are commuting contractive matrices with no eigenvalues of modulus 1,  Agler and McCarthy  proved, in their remarkable paper  \cite{AM}, that the pair $(T_1, T_2)$ has a  co-isometric  extension $(M_z^*, M_\Phi^*)$   on $H^2\otimes \CC^d$  and, for any polynomial $p$ in two variables,
\begin{equation*}
\|p(T_1, T_2)\|\leq \|p(M_z\otimes I_{\CC^{d}}, M_\Phi)\|\leq \|p\|_V,
\end{equation*}
where $V$ is a distinguished variety in $\DD^2$ depending on $T_1$ and $T_2$.

Let $f\in \CC\left[z\right]$ be a positive regular  polynomial in one variable and let $T\in \cD_f(\CC^n)$   be  an $n\times n$  matrix  which is pure with respect to the regular domain $\cD_f$, i.e. $\text{\rm SOT-}\lim \limits_{m\to\infty} \Phi_{f,{\bf T}}^m(I)=0$.
Let  $m_{T}(z)=(z-\lambda_1)^{n_1}\cdots (z-\lambda_k)^{n_k}$ be the minimal polynomial of $T$ and  let $J_{m_T}$ be  the $WOT$-closed two sided ideal of the Hardy algebra $F_1^\infty(\cD_f)$ generated by $m_{T}({\bf S})$, where ${\bf S}$ is the weighted shift associated with the domain $\cD_f$.   Note that the variety $\cV_{J_{m_T}}(\CC^n)$ consists of all $n\times n$ matrices $X$  such that $m_T(X)=0$. On the other hand,
$B:=P_{\cN_{J_{m_T}}} {\bf S}|_{\cN_{J_{m_T}}}$ is the universal model of the variety $\cV_{J_{m_T}}$, and the ellipsoid $\cE_f(\CC^n)$ is a  matrix-valued ball.
In this case, the analytic operators with respect to $B$ are the elements $\varphi(B)$ of the Hardy algebra $R_1^\infty(\cV_{J_{m_T}})$.

\begin{theorem} \label{last} Let $T_1$ and $T_2$ be commuting   matrices which are pure elements in $\cD_f(\CC^n)$ and $\cD_g(\CC^n)$     and
  let $B_1$ and $B_2$ be their universal models, respectively. If  $d_j:=\dim (I-T_jT_j^*)^{1/2}\cH$, $j=1,2$, then there exist matrix-valued analytic operators $\varphi_1(B_1)\in \cE_f(\cN_{J_{m_{T_1}}}\otimes \CC^{d_1})$ with respect to $B_1$ and    $\varphi_2(B_2)\in \cE_g(\cN_{J_{m_{T_2}}}\otimes \CC^{d_2})$ with respect to $B_2$, such that
$$
\|[p_{rs}(T_1, T_2)]_k\|\leq \min\left\{ \|[p_{rs}(B_1\otimes I_{\CC^{d_1}},\varphi_1(B_1))]_k\|, \|[p_{rs}(\varphi_2(B_2), B_2\otimes I_{\CC^{d_2}})]_k\|\right\}, \qquad
$$
for any   $[p_{rs}]_k\in M_k(\CC[z, w])$.
 \end{theorem}
 \begin{proof}
 Applying Theorem \ref{ando1} to the pairs
$({T}_1, {T}_2)\in {\bf D}_{(f,g)}^{{J_{m_{T_1}}}}(\cH)$ and $({ T}_2, { T}_1)\in {\bf D}_{(g,f)}^{{J_{m_{T_2}}}}(\cH)$, the result follows.
 \end{proof}

\smallskip

      %\Refs
      %\widestnumber\key{BFPQR}
      %\def\n{\key}
       %

\end{document}